\documentclass[draft]{amsart}

\usepackage{amssymb}                
\usepackage{fullpage}               
\usepackage{url}                    
\usepackage{color,cite}
\usepackage[all,knot,arc]{xy}

\theoremstyle{plain}
\newtheorem{lemma}{Lemma}[section]
\newtheorem{theorem}[lemma]{Theorem}
\newtheorem{prop}[lemma]{Proposition}
\newtheorem{corollary}[lemma]{Corollary}
\newtheorem{prob}[lemma]{Problem}
\newtheorem{conj}[lemma]{Conjecture}

\theoremstyle{definition}
\newtheorem{defi}[lemma]{Definition}
\newtheorem{definition}[lemma]{Definition}
\newtheorem*{definition*}{Definition}       
\newtheorem{nota}[lemma]{Notation}
\newtheorem{rem}[lemma]{Remark}
\theoremstyle{remark}
\newcommand{\Bi}{\mathcal{B}}   
\newcommand{\relsigma}{\mathrel{\sigma}}  

\DeclareMathOperator{\greenL}{\mathcal{L}}
\DeclareMathOperator{\greenR}{\mathcal{R}}
\DeclareMathOperator{\greenJ}{\mathcal{J}}
\DeclareMathOperator{\greenH}{\mathcal{H}}
\DeclareMathOperator{\greenD}{\mathcal{D}}

\newcommand{\cgp}{\mathrel{\sim_G}}
\newcommand{\con}{\mathrel{\sim_c}}
\newcommand{\ci}{\mathrel{\sim_i}}
\newcommand{\cp}{\mathrel{\sim_p}}
\newcommand{\cl}{\mathrel{\sim_l}}
\newcommand{\co}{\mathrel{\sim_{\!o}}}
\newcommand{\cu}{\mathrel{\sim_u}}
\newcommand{\ncon}{\mathrel{\sim_{\mathfrak{n}}}}
\newcommand{\ctr}{\mathrel{\sim_{tr}}}

\newcommand{\inv}{^{-1}}

\DeclareMathOperator\sym{Sym}
\DeclareMathOperator\dom{dom}
\DeclareMathOperator\ima{im}
\DeclareMathOperator\spa{span}

\DeclareMathOperator\rank{rank}

\newcommand\mz{\diamond}

\newcommand\al{\alpha}
\newcommand\bt{\beta}

\newcommand\sig{\sigma}
\newcommand\del{\delta}
\newcommand\gam{\gamma}
\newcommand\lam{\lambda}

\newcommand\vep{\varepsilon}
\newcommand\tet{\theta}
\newcommand\ups{\upsilon}

\newcommand\ome{\omega}

\newcommand\sm{\setminus}
\newcommand\lan{\langle}
\newcommand\ran{\rangle}
\newcommand\arr{\rightarrow}
\newcommand\ara{\stackrel{\al}{\arr}}
\newcommand\arb{\stackrel{\bt}{\arr}}

\newcommand\join{\bigsqcup}
\newcommand\jo{\sqcup}

\newcommand\mi{\mathcal{I}}
\newcommand\fr{\mathcal{FI}}

\newcommand\da{\Delta_\al}
\newcommand\ta{\Theta_\al}
\newcommand\oa{\Omega_\al}
\newcommand\dka{\da^k}
\newcommand\tka{\ta^k}
\newcommand\ua{\Upsilon_{\!\al}}
\newcommand\la{\Lambda_\al}
\newcommand\db{\Delta_\bt}
\newcommand\tb{\Theta_\bt}
\newcommand\ob{\Omega_\bt}
\newcommand\dkb{\db^k}
\newcommand\tkb{\tb^k}
\newcommand\ub{\Upsilon_{\!\bt}}
\newcommand\lb{\Lambda_\bt}

\newcommand\gd{\mathcal{D}}
\newcommand\gl{\mathcal{L}}
\newcommand\gr{\mathcal{R}}

\newcommand{\ale}{\aleph}
\newcommand{\aep}{\ale_\vep}
\newcommand\jr{J_r}

\newcommand\pp{\mathbb{P}}

\title{Conjugacy in Inverse Semigroups}

\author{Jo\~{a}o Ara\'{u}jo}

\author{Michael Kinyon${}^\dagger$}
\thanks{${}^\dagger$ Partially supported by Simons Foundation Collaboration Grant 359872}

\author{Janusz Konieczny}

\address[Ara\'{u}jo]{Centro de \'{A}lgebra \\
Universidade de Lisboa \\
1649--003 Lisboa \\ Portugal\\
and\\
Universidade Aberta\\
1269--001 Lisboa \\ Portugal}
\email{\url{jaraujo@ptmat.fc.ul.pt}}

\address[Kinyon]{Department of Mathematics \\
University of Denver \\ Denver, CO 80208 USA}
\email{\url{mkinyon@du.edu}}

\address[Konieczny]{Department of Mathematics \\
University of Mary Washington\\ Fredericksburg, VA 22401 USA}
\email{\url{jkoniecz@umw.edu}}

\date{}

\begin{document}

\begin{abstract}
In a group $G$, elements $a$ and $b$ are conjugate if there exists $g\in G$ such that $g\inv ag=b$.
This conjugacy relation, which plays an important role in group theory, can be extended in a natural way
to inverse semigroups: for elements $a$ and $b$ in an inverse semigroup $S$, $a$ is conjugate to $b$, which we will
write as $a\ci b$, if there exists $g\in S^1$ such that $g\inv ag=b$ and $gbg\inv =a$.
The purpose of this paper is to study the conjugacy $\ci$ in several classes of inverse semigroups:
symmetric inverse semigroups, free inverse semigroups, McAllister $P$-semigroups, factorizable inverse monoids,
Clifford semigroups, the bicyclic monoid and stable inverse semigroups.
\end{abstract}



\maketitle

\section{Introduction}
\label{sec:int}

The conjugacy relation $\cgp$ in a group $G$ is defined as follows: for $a,b\in G$, $a\cgp b$ if there exists $g\in G$ such that $g\inv ag=b$ and $b=gag\inv$.

A semigroup $S$ is said to be \emph{inverse} if for every $x\in S$, there exists exactly one $x\inv\in S$ such that $x=xx\inv x$ and $x\inv xx\inv$. Thus one can extend the definition of group conjugacy verbatim to inverse semigroups. (As usual $S^1$ denotes the semigroup $S$ extended by an identity element when none is present.) 

\begin{definition}\label{def:conjugation}
Let $S$ be an inverse semigroup. Elements $a,b\in S$ are said to be \emph{conjugate}, denoted $a\ci b$, if there exists $g\in S^1$ such that $g\inv ag = b$ and $gbg\inv = a$. In short,
\begin{equation}\label{eq:inv_conj}
a\ci b\quad\iff\quad \exists_{g\in S^1}\ (\,g\inv ag = b\,\text{ and }\,gbg\inv = a\,)\,.
\end{equation}
We call the relation $\ci$ $i$-\emph{conjugacy} (``$i$'' for ``inverse'').
\end{definition}
\smallskip

At first glance, this notion of conjugacy for inverse semigroups seems simultaneously both natural and naive: natural because it is an obvious way to extend $\cgp$ formally to inverse semigroups, and naive because one might not initially
expect a formal extension to exhibit much structure. But surprisingly, it turns out that this conjugacy coincides with one that 
Mark Sapir considered the best notion for inverse semigroups. The aim of this paper is to carry out an in depth study of $\ci$, and to show that its naturality goes far beyond its definition; $\ci$ is, in fact, as hinted by Sapir, a highly structured and interesting notion of conjugacy.

Our first three results (Section~\ref{sec:sym}) generalize the known theorems for permutations on a set $X$ to partial injective
transformations on~$X$.
\begin{itemize}
  \item[(1)] Elements $\al$ and $\bt$ of the symmetric inverse semigroup $\mi(X)$
are conjugate if and only if they have the same cycle-chain-ray type (Theorem~\ref{tcon}).
 \item[(2)] If $X$ is a finite set with $n$ elements, then $\mi(X)$ has $\sum_{r=0}^np(r)p(n-r)$ $i$-conjugacy classes (Theorem~\ref{tnuf}).
 \item[(3)] If $X$ is an infinite set with $|X|=\aep$, then $\mi(X)$ has
$\kappa^{\ale_0}$ $i$-conjugacy classes, where $\kappa=\ale_0+|\vep|$ (Theorem~\ref{tcci}).
\end{itemize}

The next two results (Section~\ref{sec:free}) concern the free inverse semigroup.
\begin{itemize}
  \item[(4)] For every element $w$ in the free inverse semigroup $\fr(X)$ generated by the set $X$, the conjugacy class of $w$ is finite
(Theorem~\ref{tfr1}).
  \item[(5)] The conjugacy problem is decidable for $\fr(X)$ (Theorem~\ref{tfr2}).
\end{itemize}

The following results (Sections~\ref{sec:factorizable}, \ref{sec:P}, \ref{sec:clifford} and \ref{sec:stable}) concern other important classes of inverse semigroups.

\begin{itemize}
  \item[(6)] If $S=P(G,\mathcal{X},\mathcal{Y})$ is a McAlister $P$-semigroup, then
for all $(A,g),(B,h)\in S$,
$(A,g)\!\ci\!(B,h)$ if and only if there exists $(C,k)\in S$ such that
{\rm(i)} $A=kB=C\land gC\land A$ and \rm{(ii)} $g=khk\inv $ (Theorem~\ref{tmca}).
  \item[(7)] If $S$ is an inverse semigroup, then for all $a,b\in S$, the set of
  all $g\in S^1$ such that $g\inv ag=b$ and $gbg\inv = a$ is upward closed in the natural partial
  order on $S^1$ (Theorem~\ref{thm:upward}).
  \item[(8)] If $S$ is a factorizable inverse monoid, then two elements are $\ci$-related if and only if they are conjugate under a unit  element. (Corollary~\ref{cor:factorizable}).
  \item[(9)] If $G$ is a group, then for any $Ha,Kb$ in the coset monoid $\mathcal{CM}(G)$,
  $Ha\!\ci\!Kb$ if and only if there exists $g\in G$ such that $g\inv Hag=Kb$ and $gKbg\inv = Ha$
  (Corollary~\ref{cor:coset}).
  \item[(10)] Elements $a$ and $b$ in a Clifford semigroup $S$ are conjugate if and only if $a$ and $b$ are group conjugate
in some subgroup of $S$ (Theorem~\ref{thm:Clifford_H}).
  \item[(11)] If $S$ is an inverse semigroup, then $S$ is a Clifford semigroup if and only if
no two different idempotents of $S$ are conjugate (Theorem~\ref{thm:Clifford_char}).
  \item[(12)] In the bicyclic monoid $\Bi$, $i$-conjugacy coincides with the minimal group congruence (Theorem~\ref{thm:bicyclic}).
  \item[(13)] An inverse semigroup $S$ is stable if and only if $i$-conjugacy and the natural partial order intersect trivially as
relations in $S\times S$ (Theorem~\ref{thm:stable}).
\end{itemize}

We give now a quick overview of the state of the art regarding notions of conjugacy for semigroups. As recalled above, the conjugacy relation $\cgp$ in a group $G$ is defined by
\begin{equation}\label{eq:group_conj}
a\cgp b\quad\iff\quad \exists_{g\in G}\ (\,g\inv ag = b\,\text{ and }\,gbg\inv = a\,)\,.
\end{equation}
(In fact, $\cgp$ is traditionally defined less symmetrically, but the symmetric form of \eqref{eq:group_conj} follows since $g\inv ag = b$ if and only if $gbg\inv = a$.) This definition does not make sense in general semigroups, so conjugacy has been generalized to semigroups in a variety of ways.

For a monoid $S$ with identity element $1$, let $U(S)$ denote its group of units. \emph{Unit conjugacy} in $S$ is modeled on $\cgp$ in groups by
\begin{equation}\label{eq:unit_conj}
  a\cu b\quad\iff\quad \exists_{g\in U(S)}\ (\,g\inv ag = b\,\text{ and }\,gbg\inv = a\,)\,.
\end{equation}
(We will write ``$\sim$'' with various subscripts for possible definitions of conjugacy in semigroups. In this case, the subscript $u$ stands for ``unit.'') See, for instance, \cite{KuMa07,KuMa09}. However, $\cu$ does not make sense in a semigroup without an identity element. Requiring $g\in U(S^1)$ in unit conjugacy does not help for arbitrary semigroups because $U(S^1) = \{1\}$ if $S$ is not a monoid.

Conjugacy in a group $G$ can, of course, be rewritten without using inverses:
$a\cgp b\in G$ are conjugate if and only if there exists $g\in G$ such that $ag=gb$. Using this
formulation, \emph{left conjugacy} $\cl$ has been defined for a semigroup $S$ \cite{Ot84,Zh91,Zh92}:
\begin{equation}
\label{eq:left_con}
a\cl b\quad\iff\quad \exists_{g\in S^1}\ ag=gb.
\end{equation}
In a general semigroup $S$, the relation $\cl$ is reflexive and transitive, but not symmetric.
In addition, if $S$ has a zero, then $\cl$ is the universal relation $S\times S$, so $\cl$ is not useful for such semigroups.

The relation $\cl$, however, is an equivalence on any free semigroup. Lallement \cite{La79} defined two elements of a free semigroup to be conjugate if they are related by $\cl$, and then showed that $\cl$ is equal to the following relation in a free semigroup $S$:
\begin{equation}\label{econ2}
a\cp b\quad\iff\quad \exists_{u,v\in S^1}\ (\,a=uv\,\text{ and }\,b=vu\,).
\end{equation}
In a general semigroup $S$, $\cp\,\neq\,\cl$ and in fact, the relation $\cp$ is reflexive and symmetric, but not necessarily transitive. Kudryavtseva and Mazorchuk \cite{KuMa07,KuMa09} considered the transitive closure $\sim_p^*$ of $\sim_p$ as a conjugacy relation in a general semigroup. (See also \cite{Hi06}.)

Otto \cite{Ot84} studied the relations $\cl$ and $\cp$ in the monoids $S$ presented by finite Thue systems, and then symmetrized $\cl$ to give yet another definition of conjugacy in such an $S$:
\begin{equation}\label{econ3}
a\co b\quad\iff\quad \exists_{g,h\in S^1}\ (\,ag=gb\,\text{ and }\,bh=ha\,)\,.
\end{equation}
The relation $\co$ is an equivalence relation in an arbitrary semigroup $S$, but, again, it is the universal relation for any semigroup with zero.

This deficiency of $\co$ was remedied in \cite{AKM}, where the following relation was defined on an arbitrary semigroup $S$:
\begin{equation}\label{econc}
a\con b\quad\iff\quad \exists_{g\in\pp^1(a)}\ \exists_{h\in\pp^1(b)}\ (\,ag=gb\,\text{ and }\,bh=ha\,),
\end{equation}
where for $a\ne0$, $\pp(a)=\{g\in S^1:\forall_{m\in S^1}\,(ma\ne0\Rightarrow(ma)g\ne0)\}$, $\pp(0)=\{0\}$, and $\pp^1(a)=\pp(a)\cup\{1\}$.
(See \cite[\S2]{AKM} for a motivation for this definition.) The relation $\con$ is an equivalence on $S$, it does not reduce to $S\times S$ if $S$ has a zero, and it is equal to $\co$ if $S$ does not have a zero.

In 2018, the third author \cite{Ko18} defined a conjugacy $\ncon$ on any semigroup $S$ by
\begin{equation}\label{e1dcon}
a\ncon b\quad\iff\quad \exists_{g,h\in S^1}\ (\,ag=gb,\,\, bh=ha,\,\, hag=b,\,\textnormal{ and }\,gbh=a\,)\,.
\end{equation}
The relation $\ncon$ is an equivalence relation on any semigroup and it does not reduce to $S\times S$ if $S$ has a zero. In fact, it is the smallest of all conjugacies defined up to this point for general semigroups.

We point out that each of the relations \eqref{eq:unit_conj}--\eqref{e1dcon} reduces to group conjugacy when $S$ is a group. However, assuming we require conjugacy to be an equivalence relation on general semigroups, only $\sim_p^*$, $\co$, $\con$, and $\ncon$ can provide possible definitions of conjugacy. We have
\[
\ncon\,\,\subseteq\,\,\sim_p^*\,\,\subseteq\,\,\co\text{ and }\ncon\,\,\subseteq\,\,\sim_c\,\,\subseteq\,\,\co,
\]
and, with respect to inclusion, $\sim_p^*$ and $\sim_c$ are not comparable \cite[Prop.~2.3]{Ko18}. For detailed comparison and analysis, in various classes of semigroups,
of the conjugacies $\sim_p^*$, $\co$, $\con$, and also trace (character) conjugacy $\ctr$ defined for epigroups, see \cite{ArKiKoMa17}.

A notion of conjugacy for inverse semigroups equivalent to our $\ci$ has appeared elsewhere.
In fact, part of our motivation for the present study was a \textsc{MathOverflow} post by Sapir \cite{Sapir}, in which he claimed that the following is the best notion of conjugacy in inverse semigroups: for $a,b$ in an inverse semigroup $S$, $a$ is conjugate to $b$ if there exists
$t\in S^1$ such that
\begin{equation}\label{eq:sapir}
t\inv at=b,\,\,\, a\cdot tt\inv = tt\inv\cdot a=a,\,\,\,\text{and}\,\,\, b\cdot t\inv t = t\inv t\cdot b=b.
\end{equation}
Sapir notes that this notion of conjugacy is implicit in the work of Yamamura~\cite{yam}.
It is easy to show that Sapir's relation coincides with $\ci$ (Proposition~\ref{prp:sapir}).

Of the conjugacies $\sim_p^*$, $\co$, $\con$, and $\ncon$ defined for an arbitrary semigroup $S$, only $\ncon$ reduces to $\ci$ if $S$ is an inverse semigroup \cite[Thm.~2.6]{Ko18}.
Observe also that if $S$ is an inverse monoid,
\begin{equation}\label{eq:unit_i}
  \cu\ \subseteq\ \ci\,.
\end{equation}
This inclusion is generally proper, but we will see that equality holds in factorizable inverse
monoids (Corollary~\ref{cor:factorizable}).

We conclude this introduction with three general results about $i$-conjugacy.
In an inverse semigroup $S$, both of the following identities hold: for all $x,y\in S$,
\[
(x\inv)\inv = x\,\,\text{ and }\,\, (xy)\inv = y\inv x\inv\,.
\]
From these, the following is easy to see.

\begin{lemma}\label{lem:equiv}
The conjugacy $\ci$ is an equivalence relation in any inverse semigroup.
\end{lemma}
\begin{proof}
Let $S$ be an inverse semigroup. Then $\ci$ is reflexive
(since $1\in S^1$) and symmetric (since $(g\inv )\inv = g$ for every $g\in S$).
For all $a,b,c\in S$, $g,h\in S^1$, if
$g\inv ag = b$, $gbg\inv = a$, $h\inv bh = c$, and $hch\inv = b$, then $(gh)\inv\cdot a\cdot gh = c$ and
$gh\cdot c\cdot(gh)\inv = a$. Thus $\ci$ is also transitive.
\end{proof}

For an inverse semigroup $S$, the equivalence class of $a\in S$ with respect to $\ci$ will be called the \emph{conjugacy class} of $a$ and denoted by $[a]_{\ci}$.

The following proposition shows that \eqref{eq:inv_conj} is equivalent to Sapir's formulation
\eqref{eq:sapir}, and also shows the specific connection between the conjugacies $\ci$ and $\co$. For an inverse semigroup $S$, $a,b\in S$ and $g\in S^1$, we consider the following equations.
\begin{center}
  \begin{tabular}{rccrc}
  (i)   & $g\inv ag = b$       & \qquad & (ii)   & $gbg\inv = a$ \\
  (iii) & $ag = gb$            & \qquad & (iv)   & $bg\inv = g\inv a$ \\
  (v)   & $a\cdot gg\inv = a$  & \qquad & (vi)   & $gg\inv\cdot a = a$ \\
  (vii) & $b\cdot g\inv g = b$ & \qquad & (viii) & $g\inv g\cdot b = b$
\end{tabular}
\end{center}

\begin{prop}
\label{prp:sapir}
Let $S$ be an inverse semigroup. For $a,b\in S$ and $g\in S^1$, the following sets of conditions are equivalent and each set implies all of (i)--(viii).
\begin{itemize}
  \item[\rm(a)]\quad $\{\mathrm{(i),(ii)}\}$ \textup{(}that is, $a\ci b$\textup{)};
  \item[\rm(b)]\quad $\{\mathrm{(i),(v),(vi)}\}$;
  \item[\rm(c)]\quad $\{\mathrm{(iii),(v),(viii)}\}$;
  \item[\rm(d)]\quad $\{\mathrm{(ii),(vii),(viii)}\}$;
  \item[\rm(e)]\quad $\{\mathrm{(iv),(vi),(vii)}\}$.
\end{itemize}
\end{prop}
\begin{proof}
(a)$\implies$(b): $a\cdot gg\inv = gbg\inv gg\inv = gbg\inv = a$ and
$gg\inv \cdot a = gg\inv gbg\inv = gbg\inv = a$.

(b)$\implies$(c): $ag = gg\inv\cdot ag = gb$ and $g\inv g\cdot b = g\inv g\cdot g\inv ag
= g\inv ag = b$.

(c)$\implies$(a): $g\inv ag = g\inv g\cdot b = b$ and $gbg\inv = a\cdot gg\inv = a$.

The cycle of implications (a)$\implies$(d)$\implies$(e)$\implies$(a) follows from the cycle
already proven by exchanging the roles of $a$ and $b$ and replacing $g$ with $g\inv$
(since $(g\inv)\inv = g$).
\end{proof}

Finally, we characterize one of the two extreme cases for $i$-conjugacy  on an inverse semigroup $S$, namely where $\ci$ is the universal relation $S\times S$. In  Theorem \ref{thm:commutative} we will consider the opposite extreme, where $\ci$ is the identity relation (equality). Similar discussions for other notions of conjugacy can be found in \cite{ArKiKoMa17}.

For an inverse semigroup $S$, we denote by $E(S)$ the semilattice of idempotents of $S$ \cite[p.~146]{Ho95}.

\begin{theorem}\label{thm:universal}
  Let $S$ be an inverse semigroup. Then $\ci$ is the universal relation $S\times S$ if and
  only if $S$ is a singleton.
\end{theorem}
\begin{proof}
  Suppose $\ci$ is universal. For all $e\in E(S)$ and $g\in S$, we have $g\inv eg\in E(S)$,
  and so every element of $S$ is an idempotent, that is, $S$ is a semilattice. Now for
  $e,f\in S$, let $g\in S$ be given such that $g\inv eg = f$ and $gfg\inv = e$. Since
  $g\inv = g$, we have $f = geg=egg = eg$ and so $e = gfg=fgg = fg = (eg)g = eg = f$.
  Therefore $S$ has only one element. The converse is trivial.
\end{proof}

\section{Conjugacy in symmetric inverse semigroups}
\label{sec:sym}
\setcounter{equation}{0}
For a nonempty set $X$ (finite or infinite), denote by $\mi(X)$ the \emph{symmetric inverse semigroup} on $X$,
that is, the semigroup of partial injective transformations on $X$ under composition. The semigroup $\mi(X)$ is
universal for the class of inverse semigroups (see \cite{Pe84} and \cite[Ch.~5]{Ho95})
since every inverse semigroup can be embedded in some $\mi(X)$ \cite[Thm.~5.1.7]{Ho95}.
This is analogous to the fact that every group can be embedded in some symmetric
group $\sym(X)$ of permutations on a set $X$. The semigroup $\mi(X)$ has $\sym(X)$ as its group
of units and contains a zero (the empty transformation, which we will denote by $0$).

In this section, we will describe conjugacy in $\mi(X)$ and its ideals, and count the conjugacy classes in $\mi(X)$
for both finite and infinite $X$.

\subsection{Cycle-chain-ray decomposition of elements of $\mi(X)$}

The cycle decomposition of a permutation can be extended to the cycle-chain-ray decomposition
of a partial injective transformation (see \cite{Ko13}).

We will write functions on the right and compose from left to right; that is,
for $f:A\to B$ and $g:B\to C$, we will write $xf$, rather than $f(x)$, and
$x(fg)$, rather than $g(f(x))$.
Let $\al\in \mi(X)$. We denote the domain of $\al$ by $\dom(\al)$ and the image of $\al$ by
$\ima(\al)$. The union $\dom(\al)\cup\ima(\al)$ will be called the \emph{span} of $\al$
and denoted $\spa(\al)$.
We say that $\al$ and $\bt$ in $\mi(X)$ are \emph{completely disjoint} if $\spa(\al)\cap\spa(\bt)=\emptyset$.
For $x,y\in X$, we write $x\ara y$ if $x\in\dom(\al)$ and $x\al=y$.

\begin{defi}\label{djoi}
Let $M$ be a set of pairwise completely disjoint elements of $\mi(X)$. The \emph{join} of the elements of $M$,
denoted $\join_{\gam\in M}\gam$, is the element of $\mi(X)$ whose domain is $\bigcup_{\gam\in M}\dom(\gam)$
and whose values are defined by
\[
x(\join_{\gam\in M}\gam)=x\gam_0,
\]
where $\gam_0$ is the (unique) element of $M$ such that $x\in\dom(\gam_0)$. If $M=\emptyset$,
we define $\join_{\gam\in M}\gam$ to be~$0$ (the zero in $\mi(X)$).
If $M=\{\gam_1,\gam_2,\ldots,\gam_k\}$ is finite, we may write the join as $\gam_1\jo\gam_2\jo\cdots\jo\gam_k$.
\end{defi}

\begin{defi}\label{dbas}
Let $\ldots,x_{-2},x_{-1},x_0,x_1,x_2,\ldots$ be pairwise distinct elements of $X$. The following elements of $\mi(X)$
will be called \emph{basic} partial injective transformations on $X$.
\begin{itemize}
  \item A \emph{cycle} of length $k$ ($k\geq1$), written $(x_0\,x_1\ldots\, x_{k-1})$,
is an element $\del\in\mi(X)$ with $\dom(\del)=\{x_0,x_1,\ldots,x_{k-1}\}$, $x_i\del=x_{i+1}$ for all $0\leq i<k-1$,
and $x_{k-1}\del=x_0$.
  \item A \emph{chain} of length $k$ ($k\geq1$), written $[x_0\,x_1\ldots\, x_k]$,
is an element $\tet\in\mi(X)$ with $\dom(\tet)=\{x_0,x_1,\ldots,x_{k-1}\}$ and $x_i\tet=x_{i+1}$ for all $0\leq i\leq k-1$.
  \item A \emph{double ray}, written $\lan\ldots\,x_{-1}\,x_0\,x_1\ldots\ran$,
is an element $\ome\in\mi(X)$ with $\dom(\omega)=\{\ldots,x_{-1},x_0,x_1,\ldots\}$ and $x_i\ome=x_{i+1}$ for all $i$.
  \item A \emph{right ray}, written $[x_0\,x_1\,x_2\ldots\ran$,
is an element $\ups\in\mi(X)$ with $\dom(\ups)=\{x_0,x_1,x_2,\ldots\}$ and $x_i\ups=x_{i+1}$ for all $i\geq0$.
  \item A \emph{left ray}, written $\lan\ldots\, x_2\,x_1\,x_0]$,
is an element $\lam\in\mi(X)$ with $\dom(\lam)=\{x_1,x_2,x_3,\ldots\}$ and $x_i\lam=x_{i-1}$ for all $i>0$.
\end{itemize}
By a \emph{ray} we will mean a double, right, or left ray.
\end{defi}

We note the following.
\begin{itemize}
  \item The span of a basic partial injective transformation is exhibited by the notation. For example, the span
of the right ray $[1\,2\,3\ldots\ran$ is $\{1,2,3,\ldots\}$.
  \item The left bracket in ``$\eta=[x\ldots$'' indicates that $x\notin\ima(\eta)$; while the right bracket
in ``$\eta=\ldots\,x]$'' indicates that $x\notin\dom(\eta)$. For example, for the chain $\tet=[1\,2\,3\,4]$,
$\dom(\tet)=\{1,2,3\}$ and $\ima(\tet)=\{2,3,4\}$.
  \item A cycle $(x_0\,x_1\ldots\, x_{k-1})$ differs from the corresponding cycle in the symmetric group
of permutations on $X$ in that the former is undefined for every $x\in(X\sm\{x_0,x_1,\ldots,x_{k-1}\})$, while the latter
fixes every such $x$.
\end{itemize}

The following decomposition was proved in \cite[Prop.~2.4]{Ko13}.

\begin{prop}
\label{pdec}
Let $\al\in\mi(X)$ with $\al\ne0$. Then there exist unique sets: $\da$ of cycles,
$\ta$ of chains, $\oa$ of double rays, $\ua$ of right rays, and $\la$ of left rays
such that
the transformations in $\da\cup\ta\cup\oa\cup\ua\cup\la$ are pairwise completely disjoint and
\begin{equation}
\al=\join_{\del\in\da}\!\!\del\jo\join_{\tet\in\ta}\!\!\tet\jo\join_{\ome\in\oa}\!\!\ome\jo\join_{\ups\in\ua}\!\!\ups
\jo\join_{\lam\in\la}\!\!\lam.\label{edec}
\end{equation}
\end{prop}

We will call the join (\ref{edec}) the \emph{cycle-chain-ray decomposition} of $\al$. If
$\eta\in\da\cup\ta\cup\oa\cup\ua\cup\la$, we will say that $\eta$ is \emph{contained} in $\al$
(or that $\al$ \emph{contains} $\eta$). If $\al=0$, we set $\da=\ta=\oa=\ua=\la=\emptyset$.
We note the following.
\begin{itemize}
  \item If $\al\in\sym(X)$, then $\al=\join_{\del\in\da}\!\del\jo\join_{\ome\in\oa}\!\ome$
(since $\ta=\ua=\la=\emptyset$), which corresponds to the usual cycle decomposition
of a permutation \cite[1.3.4]{Sc64}.
  \item If $\dom(\al)=X$, then $\al=\join_{\del\in\da}\!\del\jo\join_{\ome\in\oa}\!\ome\jo\join_{\ups\in\ua}\!\ups$
(since $\ta=\la=\emptyset$), which corresponds to the decomposition given in \cite{Le91}.
  \item If $X$ is finite, then $\al=\join_{\del\in\da}\!\del\jo\join_{\tet\in\ta}\!\tet$
(since $\oa=\ua=\la=\emptyset$), which is the decomposition given in \cite[Thm.~3.2]{Li96}.
\end{itemize}

For example, if $X=\{1,2,3,4,5,6,7,8,9\}$, then
\[
\al=\begin{pmatrix}1&2&3&4&5&6&7&8&9\\3&6&-&5&9&8&-&2&-\end{pmatrix}\in\mi(x)
\]
written in cycle-chain decomposition (no rays since $X$ is finite) is
$\al=(2\,6\,8)\jo[1\,3]\jo[4\,5\,9]$. The following~$\bt$ is an example of an element of $\mi(\mathbb Z)$ written
in cycle-chain-ray decomposition:
\[
\bt=(2\,4)\jo[6\,8\,10]\jo\lan\ldots-6\,-4\,-2\,-1\,-3\,-5\,\ldots\ran\jo[1\,5\,9\,13\,\ldots\ran\jo\lan\ldots15\,11\,7\,3].
\]

\subsection{Characterization of $\ci$ in $\mi(X)$}
We will now characterize $\ci$ in $\mi(X)$ using the cycle-chain-ray decomposition of partial
injective transformations.

\begin{nota}
\label{nzero}
We will fix an element $\mz\notin X$. For $\al\in\mi(X)$ and $x\in X$, we will write $x\al=\mz$
if and only if $x\notin\dom(\al)$. We will also assume that $\mz\al=\mz$. With this notation, it will make sense
to write $x\al=y\bt$ or $x\al\ne y\bt$ ($\al,\bt\in\mi(X)$, $x,y\in X$) even when $x\notin\dom(\al)$ or $y\notin\dom(\bt)$.
\end{nota}

\begin{lemma}\label{ltau}
Let $\al,\bt,\tau\in\mi(X)$ and suppose $\tau\inv \al\tau=\bt$ and $\tau\bt\tau\inv =\al$. Then for all $x,y\in X$:
\begin{itemize}
  \item[\rm(1)] $\spa(\al)\subseteq\dom(\tau)$ and $\spa(\bt)\subseteq\ima(\tau)$;
  \item[\rm(2)] if $x\ara y$ then $x\tau\arb y\tau$;
  \item[\rm(3)] if $x\notin\dom(\al)$ and $x\in\dom(\tau)$, then $x\tau\notin\dom(\bt)$;
  \item[\rm(4)] if $x\notin\ima(\al)$ and $x\in\dom(\tau)$, then $x\tau\notin\ima(\bt)$.
\end{itemize}
\end{lemma}
\begin{proof}
By Proposition~\ref{prp:sapir}, $\al=\tau(\tau\inv \al)$ and $\al=(\al\tau)\tau\inv $. Thus $\dom(\al)\subseteq\dom(\tau)$
and $\ima(\al)\subseteq\ima(\tau\inv )=\dom(\tau)$, and so $\spa(\al)\subseteq\dom(\tau)$. By the foregoing argument,
$\spa(\bt)\subseteq\dom(\tau\inv )=\ima(\tau)$. We have proved~(1).

To prove (2), let $x\ara y$. Since $\al\tau=\tau\bt$ (by Proposition~\ref{prp:sapir}),
$(x\tau)\bt=(x\al)\tau=y\tau$. Since $y\tau\ne\mz$ by (1), it follows that $x\tau\arb y\tau$.

To prove (3), let $x\notin\dom(\al)$ and $x\in\dom(\tau)$. Then $(x\tau)\bt=(x\al)\tau=\mz\tau=\mz$.
Thus $(x\tau)\bt=\mz$, that is, $x\tau\notin\dom(\bt)$.

To prove (4), let $x\notin\ima(\al)$ and $x\in\dom(\tau)$. Suppose to the contrary that $x\tau\in\ima(\bt)$.
Then $z\bt=x\tau$ for some $z\in X$. By Proposition~\ref{prp:sapir}, $\bt\tau\inv =\tau\inv \al$. Thus
$x=(x\tau)\tau\inv =(z\bt)\tau\inv =(z\tau\inv )\al$, and so $x\in\ima(\al)$, which is a contradiction.
Hence $x\tau\notin\ima(\bt)$.
\end{proof}

\begin{defi}
\label{dtas}
Let $\ldots,x_{-1},x_0,x_1,\ldots$ be pairwise distinct elements of $X$.
Let $\del=(x_0\ldots x_{k-1})$, $\tet=[x_0\,x_1\,\ldots\,x_k]$,
$\ome=\lan\ldots x_{-1}\, x_0\, x_1\ldots\ran$, $\ups=[x_0\, x_1\, x_2\ldots\ran$, and
$\lam=\lan\ldots x_2\,x_1\,x_0]$. For any $\eta\in\{\del,\tet,\ome,\ups,\lam\}$
and any $\tau\in\mi(X)$ such that $\spa(\eta)\subseteq\dom(\tau)$, we define $\eta\tau^*$ to be $\eta$
in which each $x_i$ has been replaced with~$x_i\tau$. Since $\tau$ is injective, $\eta\tau^*$
is a cycle of length $k$ [chain of length $k$, double ray, right ray, left ray] if $\eta$ is
a cycle of length $k$ [chain of length $k$, double ray, right ray, left ray].
For example,
\[
\del\tau^*=(x_0\tau\,x_1\tau\ldots x_{k-1}\tau)\mbox{ and }\lam\tau^*=\lan\ldots x_2\tau\,x_1\tau\,x_0\tau].
\]
\end{defi}

\begin{nota}
\label{ndka}
For $0\ne\al\in\mi(X)$, let $\da$ be the set of cycles and $\ta$ be the set of chains that occur in the
cycle-chain-ray decomposition of $\al$ (see (\ref{edec})). For $k\geq1$, we denote by $\dka$ the set of cycles in $\da$ of length $k$,
and by $\tka$ the set of chains in $\ta$ of length $k$. If $\al=0$, we set $\dka=\tka=\emptyset$.
\end{nota}

For a function $f:A\to B$ and $A_0\subseteq A$, $A_0f=\{af:a\in A_0\}$ denotes the image of $A_0$ under~$f$.

\begin{prop}
\label{pcon}
Let $\al,\bt,\tau\in\mi(X)$ be such that $\tau\inv \al\tau=\bt$ and $\tau\bt\tau\inv =\al$. Then for every $k\geq1$,
$\dka\tau^*=\dkb$, $\tka\tau^*=\tkb$, $\oa\tau^*=\ob$, $\ua\tau^*=\ub$, and $\la\tau^*=\lb$.
\end{prop}
\begin{proof}
Let $k\geq1$.
Let $\del=(x_0\,x_1\ldots\,x_{k-1})\in\dka$. Then $\del\tau^*=(x_0\tau\,x_1\tau\ldots\,x_{k-1}\tau)$.
We have $x_0\ara x_1\ara\cdots\ara x_{k-1}\ara x_0$, and so
$x_0\tau\arb x_1\tau\arb\cdots\arb x_{k-1}\tau\arb x_0\tau$ by Lemma~\ref{ltau}. Thus
$\del\tau^*\in\dkb$.
We have proved that $\dka\tau^*\subseteq\dkb$. Let $\sig=(y_0\,y_1\ldots\,y_{k-1})\in\dkb$.
By the foregoing argument, $\sig(\tau\inv )^*=(y_0\tau\inv \,y_1\tau\inv \ldots\,y_{k-1}\tau\inv )\in\dka$.
Further,
$(\sig(\tau\inv )^*)\tau^*=(y_0\tau\inv \tau\,y_1\tau\inv \tau\ldots\,y_{k-1}\tau\inv \tau)=(y_0\,y_1\ldots\,y_{k-1})=\sig$.
It follows that $\dka\tau^*=\dkb$.

Let $\tet=[x_0\,x_1\ldots\,x_k]\in\tka$. Then $\tet\tau^*=[x_0\tau\,x_1\tau\ldots\,x_k\tau]$.
We have $x_0\ara x_1\ara\cdots\ara x_k$, and so
$x_0\tau\arb x_1\tau\arb\cdots\arb x_k\tau$ by Lemma~\ref{ltau}.
Also by Lemma~\ref{ltau}, $x_0\tau\notin\ima(\bt)$
(since $x_0\notin\ima(\al)$) and $x_k\tau\notin\dom(\bt)$ (since $x_k\notin\dom(\al)$). Thus $\tet\tau^*\in\tkb$.
We have proved that $\tka\tau^*\subseteq\tkb$. Let $\eta=(y_0\,y_1\ldots\,y_{k-1})\in\tkb$.
By the foregoing argument, $\eta(\tau\inv )^*=[y_0\tau\inv \,y_1\tau\inv \ldots\,y_k\tau\inv ]\in\tka$.
Further,
$(\eta(\tau\inv )^*)\tau^*=[y_0\tau\inv \tau\,y_1\tau\inv \tau\ldots\,y_k\tau\inv \tau]=[y_0\,y_1\ldots\,y_k]=\eta$.
It follows that $\tka\tau^*=\tkb$.

The proofs of the remaining equalities are similar.
\end{proof}

\begin{defi}
\label{dtyp}
Let $\al\in\mi(X)$. The sequence
\[
\lan|\da^1|,|\da^2|,|\da^3|,\ldots;|\ta^1|,|\ta^2|,|\ta^3|,\ldots;|\oa|,|\ua|,|\la|\ran
\]
(indexed by the elements of the ordinal $2\ome+3$)
will be called the \emph{cycle-chain-ray type} of $\al$.
This notion generalizes the cycle type of a permutation \cite[p.~126]{DuFo04}.
\end{defi}

The cycle-chain-ray type of $\al$ is completely determined by the \emph{form} of the cycle-chain-ray
decomposition of $\al$. The form is obtained from the decomposition by omitting each occurrence of the symbol ``$\jo$''
and replacing each element of $X$ by some generic symbol, say ``$*$.'' For example,
$\al=(2\,6\,8)\jo[1\,3]\jo[4\,5\,9]$ has the form $(*\,*\,*)[*\,*][*\,*\,*]$, and
\[
\bt=(2\,4)\jo[6\,8\,10]\jo\lan\ldots-6\,-4\,-2\,-1\,-3\,-5\,\ldots\ran\jo[1\,5\,9\,13\,\ldots\ran\jo\lan\ldots15\,11\,7\,3]
\]
has the form
$(*\,*)[*\,*\,*]\lan\ldots *\,*\,*\,\ldots\ran[*\,*\,*\,\ldots\ran\lan\ldots *\,*\,*]$.

It is well known that two elements of the symmetric group $\sym(X)$ are conjugate if and only if
they have the same cycle type \cite[Prop.~11, p.~126]{DuFo04}. The following description
of the conjugacy in the symmetric inverse semigroup $\mi(X)$ generalizes this result.

\begin{theorem}
\label{tcon}
Elements $\al$ and $\bt$ of $\mi(X)$ are conjugate if and only if they have the same cycle-chain-ray type.
\end{theorem}
\begin{proof}
Let $\al,\bt\in\mi(X)$. Suppose $\al\ci\bt$, that is, there is $\tau\in\mi(X)$ such that
$\tau\inv \al\tau=\bt$ and $\tau\bt\tau\inv =\al$. Then $\al$ and $\bt$ have the same type by
Proposition~\ref{pcon} and the fact that $\tau^*$ restricted to any set from
$\{\dka:k\geq1\}\cup\{\tka:k\geq1\}\cup\{\oa,\ua,\la\}$ is injective.

Conversely, suppose $\al$ and $\bt$ have the same cycle-chain-ray type. Then for every $k\geq1$,
there are bijections $f_k:\dka\to\dkb$, $g_k:\tka\to\tkb$, $h:\oa\to\ob$, $i:\ua\to\ub$, and $j:\la\to\lb$.
For all $\del\in\dka$, $\tet\in\tka$, $\ome\in\oa$, $\ups\in\ua$, and $\lam\in\la$,
we define $\tau$ on $\spa(\del)\cup\spa(\tet)\cup\spa(\ome)\cup\spa(\ups)\cup\spa(\lam)$
in such a way that $\del\tau^*=\del f_k$, $\tet\tau^*=\tet g_k$, $\ome\tau^*=\ome h$, $\ups\tau^*=\ups i$,
and $\lam\tau^*=\lam j$.
Note that this defines an injective $\tau$ with $\dom(\tau)=\spa(\al)$ and $\ima(\tau)=\spa(\bt)$.

Let $x\in X$. We will prove that $x(\tau\inv \al\tau)=x\bt$. If $x\notin\spa(\bt)$ then
$x\notin\dom(\tau\inv )$ (since $\dom(\tau\inv )=\ima(\tau)=\spa(\bt)$), and so
$x(\tau\inv \al\tau)=\mz(\al\tau)=\mz$ and $x\bt=\mz$. Suppose $x\in(\ima(\bt)\setminus\dom(\bt))$. Then
there is $\xi=\ldots\,x]$ that is either a chain or left ray contained in $\bt$.
By the definition of $\tau$, there is $\eta=\ldots\,z]$ that is either a chain or left ray contained in $\al$
with $z\tau=x$. Then $x(\tau\inv \al\tau)=z(\al\tau)=\mz\tau=\mz$ and $x\bt=\mz$. Finally, suppose $x\in\dom(\bt)$.
Then there is $\xi=\ldots\,x\,y\ldots$ that is a basic partial injective transformation contained in $\bt$.
By the definition of $\tau$, there is $\eta=\ldots\,z\,w\ldots$ that is
a basic partial injective transformation contained in $\al$
with $z\tau=x$ and $w\tau=y$. Then $x(\tau\inv \al\tau)=z(\al\tau)=w\tau=y$ and $x\bt=y$.

We have proved that $\tau\inv \al\tau=\bt$. By the the same argument, applied to $\tau\inv ,\bt,\al$
instead of $\tau,\al,\bt$, we have $\tau\bt\tau\inv =\al$. Hence $\al\ci\bt$.
\end{proof}

Theorem~\ref{tcon} also follows from \cite[Cor.~5.2]{Ko18} and the fact that $\ci\,\,=\,\,\ncon$
in inverse semigroups (see Section~\ref{sec:int}). However, the proof in \cite{Ko18} is not direct since it relies on
a characterization of $\ncon$ in subsemigroups of the semigroup $P(X)$ of all partial
transformations on $X$.

Suppose $X$ is finite with $|X|=n$ and let $\al\in\mi(X)$. Then $\al$ contains no rays,
no cycles of length greater than $n$, and no chains of length greater than $n-1$. Therefore, the cycle-chain-ray type
of $\al$ can be written as
\begin{equation}\label{eccd}
\lan|\da^1|,|\da^2|,\ldots,|\da^n|;|\ta^1|,|\ta^2|,\ldots,|\ta^{n-1}|\ran.
\end{equation}
We will refer to (\ref{eccd}) as the \emph{cycle-chain type} of $\al$. By Theorem~\ref{tcon},
for all $\al,\bt\in\mi(X)$,
\begin{equation}\label{econf}
\al\ci\bt\quad\iff\quad\lan|\da^1|,\ldots,|\da^n|;|\ta^1|,\ldots,|\ta^{n-1}|\ran=
\lan|\db^1|,\ldots,|\db^n|;|\tb^1|,\ldots,|\tb^{n-1}|\ran.
\end{equation}

Suppose $\al\in\mi(X)$ has a finite domain. Then $\al$ does not contain any rays.
Therefore, we will refer to the cycle-chain-ray type of $\al$
as the cycle-chain type of $\al$ even when $X$ is infinite.

By \eqref{eq:inv_conj}, $\al$ and $\bt$ in $\mi(X)$ are conjugate if and only if there exists $\tau\in\mi(X)$
such that $\tau\inv \al\tau=\bt$ and $\tau\bt\tau\inv =\al$. If $X$ is finite, we can replace $\tau$
with a permutation on $X$.

\begin{prop}
\label{pper}
Let $X$ be a finite set, and let $\al,\bt\in\mi(X)$. Then the following conditions are equivalent:
\begin{itemize}
  \item[\rm(i)] $\al$ and $\bt$ are conjugate;
  \item[\rm(ii)] $\al$ and $\bt$ have the same cycle-chain type;
  \item[\rm(iii)] there exists $\sig\in\sym(X)$ such that $\sig\inv \al\sig=\bt$.
 \end{itemize}
\end{prop}
\begin{proof}
Conditions (i) and (ii) are equivalent by Theorem~\ref{tcon}, and (iii) clearly implies (i).
It remains to show that (i) implies (iii). Suppose (i) holds, that is, $\tau\inv \al\tau=\bt$ and $\tau\bt\tau\inv =\al$
for some $\tau\in\mi(X)$. By Proposition~\ref{pcon}, $\tau$ maps $\spa(\al)$ onto $\spa(\bt)$.
Thus $|\spa(\al)|=|\spa(\bt)|$, and so, since $X$ is finite, $|X\setminus\spa(\al)|=|X\setminus\spa(\bt)|$.
We fix a bijection $f:X\setminus\spa(\al)\to X\setminus\spa(\bt)$ and define $\sig:X\to X$ by
\[
x\sig=
\left\{\begin{array}{ll}
x\tau & \mbox{if $x\in\spa(\al)$},\\
xf & \mbox{if $x\in(X\setminus\spa(\al))$}.
\end{array}\right.
\]
Clearly, $\sig\in\sym(X)$.
Let $x\in X$. If $x\not\in\spa(\bt)$, then $x\sig\inv \notin\spa(\al)$, and so
$x(\sig\inv \al\sig)=\mz\sig=\mz=x\bt$.
Suppose $x\in(\ima(\bt)\setminus\dom(\bt))$. Then $x\tau\inv =x\sig\inv $ (by the definition of $\sig$)
and $x\tau\inv \notin\dom(\al)$ (by Lemma~\ref{ltau}). Thus,
$x(\sig\inv \al\sig)=x(\tau\inv \al\sig)=\mz\sig=\mz=x\bt$.
Suppose $x\in\dom(\bt)$. Then $x\tau\inv =x\sig\inv $ and $x\tau\inv \in\dom(\al)$.
Hence, $(x\tau\inv )\al\in\ima(\al)$, and so $((x\tau\inv )\al)\tau=((x\tau\inv )\al)\sig$.
Therefore, $x(\sig\inv \al\sig)=x(\tau\inv \al\tau)=x\bt$.

We have proved that $x(\sig\inv \al\sig)=x\bt$ for all $x\in X$, and so (i) implies (iii).
\end{proof}

The equivalence of (ii) and (iii) is stated in \cite[p.~120]{DiHa03}.
Proposition~\ref{pper} is not true for an infinite set $X$. Let $X=\{1,2,3,\ldots\}$
and consider $\al=[2\,3\,4\ldots\ran$ and $\bt=[1\,2\,3\ldots\ran$ in $\mi(X)$.
Then $\al$ and $\bt$ are conjugate by Theorem~\ref{tcon}. Note that $1\notin\dom(\al)$
and $\dom(\bt)=X$. Thus, by Lemma~\ref{ltau}(3),
if $\tau\in\mi(X)$ is such that $\tau\inv \al\tau=\bt$ and $\tau\bt\tau\inv =\al$,
then $1\not\in\dom(\tau)$. Consequently, (iii) is not satisfied.

\subsection{Conjugacy in the ideals of $\mi(X)$}
We have already dealt with the conjugacy in $\mi(X)$ (Theorem~\ref{tcon}). Here, we will describe
the conjugacy in an arbitrary proper (that is, different from $\mi(X)$) ideal of $\mi(X)$.
For $\al\in\mi(X)$, the \emph{rank} of $\al$ is the cardinality of $\ima(\al)$. Since $\al$ is injective,
we have $\rank(\al)=|\ima(\al)|=|\dom(\al)|$.
For a cardinal $r$ with $0<r\leq|X|$, let $\jr=\{\al\in\mi(X):\rank(\al)<r\}$. Then the set $\{\jr:0<r\leq|X|\}$ consists of
of all proper ideals of $\mi(X)$ \cite{Li53}.

\begin{theorem}
\label{tide1}
Let $\jr$ be a proper ideal of $\mi(X)$, where $r$ is finite, and let $\al,\bt\in\jr$.
Then $\al$ and $\bt$ are conjugate in $\jr$ if and only if they have they same cycle-chain type
and $|\spa(\al)|<r$.
\end{theorem}
\begin{proof}
Suppose $\al\ci\bt$ in $\jr$. Then $\al\ci\bt$ in $\mi(X)$, and so $\al$ and $\bt$ have the same cycle-chain
type by Theorem~\ref{tcon}. Let $\tau\in\jr$ such that
$\tau\inv \al\tau=\bt$ and $\tau\bt\tau\inv =\al$. Then, by Lemma~\ref{ltau}, $\spa(\al)\subseteq\dom(\tau)$,
and so $|\spa(\al)|\leq|\dom(\tau)|=\rank(\tau)<r$.

Conversely, suppose that $\al$ and $\bt$ have they same cycle-chain type and $|\spa(\al)|<r$.
Then $\al\ci\bt$ in $\mi(X)$ by Theorem~\ref{tcon}. In the proof of Theorem~\ref{tcon},
we constructed $\tau\in\mi(X)$ such that
$\dom(\tau)=\spa(\al)$, $\tau\inv \al\tau=\bt$, and $\tau\bt\tau\inv =\al$. Since $\rank(\tau)=|\dom(\tau)|=|\spa(\al)|<r$,
we have $\tau\in\jr$, and so $\al\ci\bt$ in~$\jr$.
\end{proof}

We note that for all $\al,\bt\in\jr$, where $r$ is finite,
\[
|\spa(\al)|=\rank(\al)+\mbox{the number of chains in $\al$},
\]
and that if $\al$ and $\bt$ have the same cycle-chain type, then $\rank(\al)=\rank(\bt)$ and $|\spa(\al)|=|\spa(\bt)|$.

As an example, let $X=\{1,\ldots,8\}$ and consider $\al=(1\,2)[3\,4][5\,6\,7]$ and $\bt=(5\,9)[1\,6][3\,8\,7]$ in $\mi(X)$.
Then $\al,\bt\in J_6$ but they are not conjugate in $J_6$ since $|\spa(\al)|=7>6$. Note, however, that $\al\ci\bt$ in $J_8$.

If $r$ is infinite, then the conjugacy $\ci$ in $J_r$ is the restriction of $\ci$ in $\mi(X)$, that is,
for all $\al,\bt\in\jr$, $\al\ci\bt$ in $\jr$ if and only if $\al\ci\bt$ in $\mi(X)$.

\begin{theorem}
\label{tide2}
Let $\jr$ be a proper ideal of $\mi(X)$, where $r$ is infinite, and let $\al,\bt\in\jr$.
Then $\al$ and $\bt$ are conjugate in $\jr$ if and only if they have they same cycle-chain-ray type.
\end{theorem}
\begin{proof}
If $\al\ci\bt$ in $\jr$, then $\al\ci\bt$ in $\mi(X)$, and so $\al$ and $\bt$ have the same cycle-chain-ray
type by Theorem~\ref{tcon}.
Conversely, suppose that $\al$ and $\bt$ have they same cycle-chain-ray type.
Then $\al\ci\bt$ in $\mi(X)$ by Theorem~\ref{tcon}. In the proof of Theorem~\ref{tcon},
we constructed $\tau\in\mi(X)$ such that
$\dom(\tau)=\spa(\al)$, $\tau\inv \al\tau=\bt$, and $\tau\bt\tau\inv =\al$. Since $\spa(\al)=\dom(\al)\cup\ima(\al)$,
we have $|\spa(\al)|\leq|\dom(\al)|+|\ima(\al)|=\rank(\al)+\rank(\al)<r+r=r$ (since $r$ is infinite).
Thus $\rank(\tau)=|\dom(\tau)|=|\spa(\al)|<r$. Hence $\tau\in\jr$,
and so $\al\ci\bt$ in $\jr$.
\end{proof}

\subsection{Number of conjugacy classes in $\mi(X)$}
We will now count the conjugacy classes in $\mi(X)$. Of course, we will have to distinguish between
the finite and infinite $X$.

Let $n$ be a positive integer. Recall that a \emph{partition} of $n$ is a sequence $\lan n_1,n_2,\ldots,n_s\ran$
of positive integers such that $n_1\leq n_2\leq\ldots\leq n_s$ and $n_1+n_2+\cdots+n_s=n$. We denote by $p(n)$
the number of partitions of $n$ and define $p(0)$ to be $1$. For example, $n=4$ has five partitions: $\lan1,1,1,1\ran$, $\lan1,1,2\ran$,
$\lan1,3\ran$, $\lan2,2\ran$, and $\lan4\ran$; so $p(4)=5$. Denote by $Q(n)$ the set of sequences
$\lan(i_1,k_1),\ldots,(i_u,k_u)\ran$ of pairs of positive integers such that
$k_1<k_2<\ldots<k_u$ and $i_1k_1+i_2k_2+\cdots+i_uk_u=n$. There is an obvious one-to-one correspondence
between the set of partitions of $n$ and the set $Q(n)$, so $|Q(n)|=p(n)$. For example,
the partition $\lan1,1,2,2,2,2,5\ran$ of $15$ corresponds to $\lan(2,1),(4,2),(1,5)\ran\in Q(15)$.
We define $Q(0)$ to be $\lan(0,0)\ran$.

\begin{nota}
\label{ncr}
Let $X$ be a finite set with $|X|=n$. Then every $\al\in\mi(X)$ can be expressed uniquely as a join $\al=\sig_\al\jo\eta_\al$,
where $\sig_\al$ is either $0$ or a join of cycles, and $\eta_\al$ is either $0$ or a join of chains.
In other words,
$\sig_\al=\join_{\del\in\da}\!\del$ and $\eta_\al=\join_{\tet\in\ta}\!\tet$. For example, if
$\al=(2\,6\,8)\jo[1\,3]\jo[4\,5\,9]$, then $\sig_\al=(2\,6\,8)$ and $\eta_\al=[1\,3]\jo[4\,5\,9]$.
Note that $|\spa(\sig_\al)|=\sum_{k=1}^nk|\dka|$ and $|\spa(\eta_\al)|=\sum_{k=1}^{n-1}(k+1)|\tka|$.

Let $C=\{[\al]_{\ci}:\al\in\mi(X)\}$ be the set of conjugacy classes of $\mi(X)$. For $r\in\{0,1,\ldots,n\}$, denote
by $C_r$ the following subset of $C$:
\begin{equation}\label{ecr}
C_r=\{[\al]_{\ci}\in C:|\spa(\sig_\al)|=r\}.
\end{equation}
By Theorem~\ref{tcon}, each $C_r$ is well defined (if $\al\ci\bt$ then $|\spa(\sig_\al)|=|\spa(\sig_\bt)|$)
and $C_0,C_1,\ldots, C_n$ are pairwise disjoint.
\end{nota}

\begin{lemma}
\label{lcr}
Let $X$ be a finite set with $n$ elements, let $r\in\{0,1,\ldots,n\}$, and let $C_r$ be the set defined by \eqref{ecr}.
Then $|C_r|=p(r)p(n-r)$.
\end{lemma}
\begin{proof}
Let $[\al]_{\ci}\in\ C_r$.
Let $K=\{k\in\{1,\ldots,n\}:\dka\ne\emptyset\}$. Write
$K=\{k_1,k_2,\ldots,k_u\}$ with $k_1<k_2<\ldots<k_u$ ($u=0$ if $K=\emptyset$). For $p\in\{1,\ldots,u\}$,
let $i_p=|\da^{k_p}|$. By (\ref{econf}), the sequence $\lan(i_1,k_1),\ldots,(i_u,k_u)\ran$
(which we define to be $\lan(0,0)\ran$ if $K=\emptyset$) does not depend
on the choice of a representative in $[\al]_{\ci}$ and
\begin{equation}\label{lcre1}
i_1k_1+\cdots+i_uk_u=\sum_{k=1}^nk|\dka|=|\spa(\sig_\al)|=r.
\end{equation}
Let $L=\{l\in\{1,\ldots,n\}:\mbox{$l\geq2$ and $\ta^{l-1}\ne\emptyset$ or $l=1$ and $X\setminus\spa(\al)\ne\emptyset$}\}$.
(The reason we include $l$ when $\ta^{l-1}\ne\emptyset$ is that there are $l$ points in the span of each chain
$[x_0\,x_1\ldots\,x_{l-1}]$ from $\ta^{l-1}$;
and we include $1$ when $X\setminus\spa(\al)\ne\emptyset$ because $X\setminus\spa(\al)$ consists of single points.)
Write
$L=\{l_1,l_2,\ldots,l_v\}$ with $l_1<l_2<\ldots<l_v$ ($v=0$ if $L=\emptyset$). For $q\in\{1,\ldots,v\}$,
let $j_q=|\ta^{l_q-1}|$ (if $l_q\geq2$) and $j_q=|X\setminus\spa(\al)|$ (if $l_q=1$).
By (\ref{econf}), the sequence $\lan(j_1,l_1),\ldots,(j_v,l_v)\ran$
(which we define to be $\lan(0,0)\ran$ if $L=\emptyset$)
does not depend
on the choice of a representative in $[\al]_{\ci}$ and
\begin{equation}\label{lcre2}
j_1l_1+\cdots+j_vl_v=\sum_{l=1}^{n-1}(l+1)|\ta^l|+|X\setminus\spa(\al)|=|\spa(\eta_\al|+|X\setminus\spa(\al)|=n-r
\end{equation}
(since $|\spa(\sig_\al)|+|\spa(\eta_\al)|+|X\setminus\spa(\al)|=n$).

Define a function $f:C_r\to Q(r)\times Q(n-r)$ (see the paragraph before Notation~\ref{ncr}) by
\[
([\al]_{\ci})f=(\lan(i_1,k_1),\ldots,(i_u,k_u)\ran,\lan(j_1,l_1),\ldots,(j_v,l_v)\ran).
\]
Then $f$ is well defined and one-to-one by (\ref{econf}), (\ref{lcre1}), and (\ref{lcre2}).
Let
\[
(\lan(i_1,k_1),\ldots,(i_u,k_u)\ran,\lan(j_1,l_1),\ldots,(j_v,l_v)\ran)\in Q(r)\times Q(n-r).
\]
Then
we can find $\al\in\mi(X)$ that has $i_p$ cycles of length $k_p$ (for each $p\in\{1,\ldots,u\}$)
and $j_q$ chains of length $l_q-1$ (for each $q\in\{1,\ldots,v\}$ such that $l_q\geq2$).
For such an $\al$, $[\al]_{\ci}\in C_r$ and
\[
([\al]_{\ci})f=(\lan(i_1,k_1),\ldots,(i_u,k_u)\ran,\lan(j_1,l_1),\ldots,(j_v,l_v)\ran),
\]
so $f$ is onto. Hence $f$ is a bijection, and so $|C_r|=|Q(r)\times Q(n-r)|=|Q(r)||Q(n-r)|=p(r)p(n-r)$.
\end{proof}

If $X$ is a finite set with $n$ elements, then the symmetric group $\sym(X)$ has
$p(n)$ conjugacy classes \cite[Prop.~11, p.~126]{DuFo04}. The following theorem,
which counts the conjugacy classes in the symmetric inverse semigroup $\mi(X)$, generalizes this result.

\begin{theorem}
\label{tnuf}
Let $X$ be a finite set with $n$ elements. Then $\mi(X)$ has $\sum_{r=0}^np(r)p(n-r)$ conjugacy classes.
\end{theorem}
\begin{proof}
Let $C$ be the set of conjugacy classes of $\mi(X)$. Then $C=C_0\cup C_1\cup\ldots\cup C_n$
and $C_0,C_1,\ldots,C_n$ are pairwise disjoint (see Notation~\ref{ncr}.)
The result follows by Lemma~\ref{lcr}.
\end{proof}

For example, if $n=5$, then the number of conjugacy classes of $\mi(X)$ is
\[
\sum_{r=0}^5p(r)p(5-r)=1\cdot7+1\cdot5+2\cdot3+3\cdot2+5\cdot1+7\cdot1=36.
\]

We will now count the conjugacy classes in $\mi(X)$ for an infinite $X$. First, we need the following lemma.
We denote by $\aep$ the infinite cardinal indexed by the ordinal $\vep$ \cite[p.~131]{HrJe99}.

\begin{lemma}
\label{laal}
Let $X$ be an infinite set with $|X|=\aep$ and let $\al\in\mi(X)$.
Then for all $k\geq1$ and all $A\in\{\dka,\tka,\oa,\ua,\la\}$,
$|A|\leq\aep$.
\end{lemma}
\begin{proof}
Suppose $A=\oa$.
Let $Z=\bigcup_{\ome\in\oa}\spa(\ome)\subseteq X$. Since the elements of $\oa$ are pairwise completely disjoint
and $|\spa(\ome)|=\ale_0$ for every $\ome\in\oa$, we have
\[
\aep=|X|\geq|Z|=|\bigcup_{\ome\in\oa}\spa(\ome)|=|\oa|\cdot\ale_0\geq|\oa|.
\]
Thus $|\oa|\leq\aep$. The proofs for the remaining values of $A$ are similar.
\end{proof}

For sets $A$ and $B$, we denote by $A^B$ the set of all functions from $B$ to $A$.

\begin{theorem}
\label{tcci}
Let $X$ be an infinite set with $|X|=\aep$. Let $\kappa=\ale_0+|\vep|$. Then $\mi(X)$ has
$\kappa^{\ale_0}$ conjugacy classes.
\end{theorem}
\begin{proof}
Let $M$ be the set of all cardinals $\mu$ such that $\mu\leq\aep$. Then $M$ consists of $\ale_0$
finite cardinals and $|\vep|+1$ infinite cardinals, hence $|M|=\ale_0+|\vep|+1=\ale_0+|\vep|=\kappa$.
Let $C$ be the set of conjugacy classes of $\mi(X)$.
Define a function $f:C\to M^{\mathbb N}$, where $\mathbb N=\{1,2,3,\ldots\}$, by
\[
([\al)]_{\ci})f=\lan|\oa|,|\ua|,|\la|,|\da^1|,|\ta^1|,|\da^2|,|\ta^2|,|\da^3|,|\ta^3|,\ldots\ran.
\]
By Theorem~\ref{tcon} and Lemma~\ref{laal},
$f$ is well defined and one-to-one. Thus $|C|\leq|M^{\mathbb N}|=|M|^{|\mathbb N|}=\kappa^{\ale_0}$.

We next define a one-to-one
function $g:M^{\mathbb N}\to C$. Let
\[
p=\lan\mu_1,\mu_2,\mu_3,\ldots\ran\in M^{\mathbb N}.
\]
Let $\mu=\sum_{k=1}^\infty k\mu_k$
(see \cite[Ch.~9]{HrJe99}). For every $k\geq1$, $k\mu_k\leq\aep$ (since $\mu_k\leq\aep$ and $\aep$ is infinite). Thus
\[
\mu=\sum_{k=1}^\infty k\mu_k\leq\ale_0\cdot\aep=\aep.
\]
Hence, there is a collection $\{X_k\}_{k\geq1}$ of pairwise disjoint subsets of $X$ such that
$|X_k|=k\mu_k$ for every $k\geq1$. Let $k\geq1$. Since $|X_k|=k\mu_k$, there is a collection
$\Delta^{\!k}$ of $k$-cycles in $\mi(X)$ such that $|\Delta^{\!k}|=\mu_k$ and $\spa(\join_{\del\in\Delta^k}\!\del)=X_k$.
Let $\al_k=\join_{\del\in\Delta^k}\!\del$ and let $\al_p=\join_{k\geq1}\!\al_k\in\mi(X)$.
We define $g:M^{\mathbb N}\to C$ by $pg=[\al_p]_{\ci}$.

Suppose $\al_p\ci\al_s$, where $p,s\in M^{\mathbb N}$. Then, by the definition of $g$, both $\al_p$ and $\al_s$
are joins of cycles and
\[
\lan|\Delta_{\al_p}^1|,|\Delta_{\al_p}^2|,|\Delta_{\al_p}^3|,\ldots\ran=\lan\mu_1,\mu_2,\mu_3,\ldots\ran=
\lan|\Delta_{\al_s}^1|,|\Delta_{\al_s}^2|,|\Delta_{\al_s}^3|,\ldots\ran.
\]
It follows from Theorem~\ref{tcon} that $g$ is one-to-one.
Hence
$|C|\geq|M^{\mathbb N}|=|M|^{|\mathbb N|}=\kappa^{\ale_0}$.
The result follows.
\end{proof}

As an example, suppose $|X|=\ale_{\ome_1}$, where $\ome_1$ is the least uncountable ordinal. Then
$\ale_0+|\ome_1|=\ale_0+\ale_1=\ale_1$, and so
the number of conjugacy classes in $\mi(X)$ is $\ale_1^{\ale_0}=2^{\ale_0}$.
(Clearly $2^{\ale_0}\leq\ale_1^{\ale_0}$. On the other hand, $\ale_1^{\ale_0}\leq(2^{\ale_0})^{\ale_0}=2^{\ale_0\ale_0}=2^{\ale_0}$.)
By a similar argument, if $|X|=\aep$, where $\vep$ is any countable ordinal or any ordinal of cardinality $\ale_1$, then
$\mi(X)$ has $2^{\ale_0}$ conjugacy classes. (The axioms of set theory cannot decide where in the
aleph hierarchy the cardinal $2^{\ale_0}$ occurs. If one assumes the Continuum Hypothesis, then $2^{\ale_0}=\ale_1$.)

\section{Conjugacy in Free Inverse Semigroups}
\label{sec:free}
\setcounter{equation}{0}
For a nonempty set $X$ (finite or infinite), denote by $\fr(X)$ the \emph{free inverse semigroup} on $X$.
In this section, we will show that for every $w\in\fr(X)$, the conjugacy class of $w$ is finite
(Theorem~\ref{tfr1}). It will then follow that the conjugacy problem in $\fr(X)$ is decidable (Theorem~\ref{tfr2}).
We also characterize those $w\in\fr(X)$ whose conjugacy class is a singleton (Proposition~\ref{psin}).

Let $X$ be a non-empty set. We say that an inverse semigroup
$F$ is a \emph{free inverse semigroup} on $X$
if it satisfies the following properties:
\begin{itemize}
\item[(1)] $X$ generates $F$;
\item[(2)] for every inverse semigroup $S$ and every mapping $\phi:X\to S$, there is an extension of $\phi$
to a homomorphism $\overline{\phi}:F\to S$.
\end{itemize}
Since $X$ generates $F$, an extension $\overline{\phi}$ is necessarily unique. It is well known
that a free inverse semigroup on $X$ exists and is unique \cite[\S5.10]{Ho95}. We will denote this unique object
by $\fr(X)$. The semigroup $\fr(X)$ can be constructed as follows \cite[Thm.~5.10.1]{Ho95}.
Let $X\inv =\{x\inv :x\in X\}$ be a set that is disjoint from $X$, let $Y=X\cup X\inv $, and let
$Y^+$ be the free semigroup on $Y$. For every $y\in Y$, we define $y\inv $ to be $x\inv $ if $y=x\in X$,
and to be $x$ if $y=x\inv \in X\inv $. Then $\fr(X)$ is isomorphic to to the quotient semigroup
$Y^+/\tau$, where $\tau$ is the smallest congruence on $Y^+$ that contains the relation
$\{(xx\inv x,x):x\in Y\}\cup\{(xx\inv yy\inv ,yy\inv xx\inv ):x,y\in Y\}$.
We will represent the congruence classes modulo $\tau$ (the elements of $\fr(X)$)
by their representatives, that is, for $w\in Y^+$, we will write $w\in\fr(X)$
instead of $w\tau\in\fr(X)$. Moreover, for $w_1,w_2\in Y^+$, we will write
$w_1=w_2$ both when $w_1\tau=w_2\tau$ (that is, when $w_1$ and $w_2$ are equal as elements of $\fr(X)$),
and when $w_1$ and $w_2$ are equal as words in $Y^+$. It should always be clear from the context
which equality is meant. For $w=x_1x_2\ldots x_n\in\fr(X)$ $(x_i\in Y)$, the unique
inverse of $w$ in $\fr(X)$ is $w\inv =x_n\inv \ldots x_2\inv x_1\inv $.

For $w\in Y^+$, we denote by $|w|$ the length of $w$ (that is, the number letters in $w$), by $i(w)$ the first letter in~$w$,
and by $t(w)$ the last letter in $w$. We say that $w$ is \emph{reduced} if it does not contain any subword $xx\inv $, where $x\in Y$.
For example, if $w=aba\inv ab\inv $, then $|w|=5$, $i(w)=a$, $t(w)=b\inv $, and $w$ is not reduced. We also consider the empty
word $1$, with $|1|=0$. In the free group on $X$ (which can be defined by $Y^*/\rho$, where $Y^*=Y^+\cup\{1\}$
and $\rho$ is the smallest congruence on $Y^*$ that contains the relation $\{(xx\inv ,1):x\in Y\}$),
each congruence class modulo $\rho$ contains exactly one reduced word \cite[p.~3]{LySc77}. The situation is more complicated
when one considers the congruence classes of $\fr(X)$. However, Poliakova and Schein \cite{PoSc05} have proved
that each congruence class of $\fr(X)$ contains a word of a certain type, which they called a canonical word,
and showed how to convert effectively any word $w\in Y^+$ to a canonical word that is in the congruence class of~$w$.
Moreover, the canonical words contained in the same congruence class are precisely the shortest words in that class.
Throughout this section, we will rely on this representation. We begin with two definitions \cite[Def.~1 and~4]{PoSc05}.

\begin{defi}
\label{dcid}
(Canonical Idempotents)
\begin{itemize}
  \item[(i)] The empty word is a \emph{canonical idempotent}, which has no \emph{factors}.
  \item[(ii)] If $e$ is a canonical idempotent, $x\in Y$, and the first letters of the factors of $e$ are different from $x$,
then $x\inv ex$ is both a canonical idempotent and a \emph{prime canonical idempotent}. This canonical idempotent
is its only factor.
  \item[(iii)] If $e_1,\ldots,e_m$, where $m\geq1$, are prime canonical idempotents and their first letters
are pairwise distinct, then $e_1\ldots e_m$ is a canonical idempotent, which has $e_1,\ldots,e_m$ as its factors.
\end{itemize}
\end{defi}
For example, if $X=\{a,b,c,\ldots\}$, then $e=(a(b\inv b)a\inv )(cc\inv )$ is a canonical idempotent
with factors $a(b\inv b)a\inv $ and $cc\inv $. (The parentheses are used for convenience only and they are not part
of the word.)

\begin{defi}
\label{dcwo}
(Canonical Words) A word $w\in Y^+$ is called a \emph{canonical word} if $w=u_0e_1u_1\ldots e_mu_m$,
where $m\geq0$, and
\begin{itemize}
  \item[(1)] $u_1,\ldots,u_{m-1}$ are not empty and $u_0\ldots u_m$ is either empty or reduced;
  \item[(2)] $e_1,\ldots,e_m$ are nonempty canonical idempotents;
  \item[(3)] for every $i\in\{1,\ldots,m\}$, the last letter of $u_{i-1}$ is different from the last letters
of the factors of $e_i$;
  \item[(4)] for every $i\in\{1,\ldots,m\}$, the first letter of $u_i$ is different from the first letters
of the factors of $e_i$.
\end{itemize}
\end{defi}
As in \cite{PoSc05}, $u_0\ldots u_m$ is called the \emph{root} of $w$, denoted by $R(w)$,
$u_0,\ldots,u_m$ are the \emph{root pieces}, and $e_1,\ldots,e_m$ the \emph{idempotent pieces} of $w$.
Whenever we write $w=u_0e_1u_1\ldots e_mu_m$, we will mean $w$ to be
in canonical form.

For example, $w=a\inv b(ab\inv ba\inv cc\inv )c\inv ab\inv a(abb\inv a\inv )ba$ is a canonical word
with three root pieces $u_0=a\inv b$, $u_1=c\inv ab\inv a$, and $u_2=ba$ and
two idempotent pieces (enclosed in parentheses) $e_1=(a(b\inv b)a\inv )(cc\inv )$ and $e_2=a(bb\inv )a\inv $.

The following lemma summarizes \cite[The Main Theorem]{PoSc05} and \cite[The Main Lemma]{PoSc05}.

\begin{lemma}\label{lmain}
~~~
\begin{itemize}
  \item[\rm(1)] Let $e$ and $f$ be canonical idempotents in $Y^+$. Then $e=f$ in $\fr(X)$ if and only if
$f$ can be obtained from $e$ by applying the operation
of commuting adjacent subwords that are canonical idempotents finitely many times.
  \item[\rm(2)] Let $w=u_0e_1u_1\ldots e_mu_m$ and $w'=v_0f_1v_1\ldots f_nv_n$ be canonical words in $Y^+$. Then
$w=w'$ in $\fr(X)$ if and only if $m=n$, $u_i=v_i$ in $Y^*$ for every $i\in\{0,1,\ldots,m\}$,
and $e_i=f_i$ in $\fr(X)$ for every $i\in\{1,\ldots,m\}$.
  \item[\rm(2)] For every $u\in Y^+$, there is a canonical word $w\in Y^+$ such that $u=w$ in $\fr(X)$. In each
congruence class of $\tau$, the canonical words and the shortest words are the same.
\end{itemize}
\end{lemma}

For example,
the canonical idempotents $e=a(bb\inv cc\inv )a\inv b\inv b$ and $f=b\inv ba(cc\inv bb\inv )a\inv $
are equal in $\fr(X)$.

Let $w\in\fr(X)$. We now establish for which letters $x\in Y$, $x\inv wx$ is conjugate to $w$.

\begin{defi}
\label{daw}
For a canonical idempotent $e\in Y^+$, we denote by $A_1(e)$ the set of the first letters of the factors of $e$,
and by $A_2(e)$ the set of the last letters of the factors of $e$. Let $w=u_0e_1u_1\ldots e_mu_m\in Y^+$
be a canonical word that is not an idempotent. We define
\begin{align}
A_1(w)&=\left\{\begin{array}{ll}\{x\}&\mbox{if $u_0\ne1$ and $x=i(u_0)$,}\\
A_1(e_1)\cup\{x\}&\mbox{if $u_0=1$ and $x=i(u_1)$,}\end{array}\right.\notag\\
A_2(w)&=\left\{\begin{array}{ll}\{x\}&\mbox{if $u_m\ne1$ and $x=t(u_m)$,}\\
A_2(e_m)\cup\{x\}&\mbox{if $u_m=1$ and $x=t(u_{m-1})$.}\end{array}\right.\notag
\end{align}
For a canonical word $w\in Y^+$ (idempotent or not), we define
\[
A(w)=\{x\in Y:\mbox{$x\in A_1(w)$ and $x\inv \in A_2(w)$}\}.
\]
Finally, for any word $u\in Y^+$ (canonical or not), we define $A(u)$ as $A(w)$,
where $w$ is a canonical word such that $u=w$ in $\fr(X)$. Note that, by Lemma~\ref{lmain},
the definition of $A(u)$ does not depend on the choice of a canonical word $w$ in the congruence class of $u$.
\end{defi}

For example, if $w=(a\inv bb\inv ac\inv c)ab(bb\inv )a\inv cc$,
then $A_1(w)=\{a\inv ,c\inv ,a\}$, $A_2(w)=\{c\}$, and so $A(w)=\{c\inv \}$.

\begin{lemma}
\label{lxxw}
Let $w=u_0e_1u_1\ldots e_mu_m\in\fr(X)$ be canonical and $x\in Y$. Then:
\begin{itemize}
  \item[\rm(1)] $xx\inv w=wxx\inv =w$ if and only if $x\in A(w)$;
  \item[\rm(2)] $x\inv wx$ is a conjugate of $w$ if and only if $x\in A(w)$.
\end{itemize}
\end{lemma}
\begin{proof}
Suppose $x\in A(w)$. Then $x\in A_1(w)$ and $x\inv \in A_2(w)$.
If $x=i(u_0)$, then $w=xv$, and so $xx\inv w=xx\inv xv=xv=w$. If $u_0=1$ and $xfx\inv $ is a factor of $e_1$,
then $xx\inv w=w$ since $e_1=xfx\inv h$, and so $xx\inv e_1=e_1$. Finally, if $u_0=1$ and $u_1=xv$, then $xx\inv w=w$ since
$xx\inv e_1xv=e_1xx\inv xv=e_1xv$. We have proved that $xx\inv w=w$ using the fact that $x\in A_1(w)$.
Similarly, $x\inv \in A_2(w)$ implies $wxx\inv =w$.

Conversely, suppose that $x\notin A(w)$, that is, $x\notin A_1(w)$
or $x\inv \notin A_2(w)$. Suppose $x\notin A_1(w)$. If $u_0\ne1$ with $x\ne i(u_0)$, then
$xx\inv w=xx\inv u_0e_1u_1\ldots e_mu_m$ has $m+1$ idempotent pieces, and so $xx\inv w\ne w$ by Lemma~\ref{lmain}.
Suppose $u_0=1$, $x$ is not the first letter of any factor of $e_1$, and $x\ne i(u_1)$.
Then $xx\inv w=xx\inv e_1u_1\ldots e_mu_m$ is canonical with the first idempotent piece $xx\inv e_1$.
By Lemma~\ref{lmain}, $e_1\ne xx\inv e_1$, and so $w\ne xx\inv w$. Similarly, if $x\inv \notin A_2(w)$,
then $wxx\inv \ne w$.

We have proved (1). Statement (2) follows from (1) and Proposition~\ref{prp:sapir}.
\end{proof}

\begin{lemma}
\label{lstr}
Let $w,w'\in\fr(X)$. Then $w\ci w'$ if and only if there are $w_0,w_1,\ldots,w_k$ in $\fr(X)$ and
$x_1,\ldots,x_k$ in $Y$, where $k\geq0$, such that $w_0=w$, $w_k=w'$, $w_0\ci w_1\ci\ldots\ci w_k$, and
$w_i=x_i\inv w_{i-1}x_i$ for every $i\in\{1,\ldots,k\}$.
\end{lemma}
\begin{proof}
Suppose $w\ci w'$. Then, there is
$u=x_1\ldots x_k\in\fr(X)$, where $k\geq0$ and $x_i\in Y$,
such that $w'=u\inv wu=x_k\inv \ldots x_1\inv wx_1\ldots x_k$ and $w=uw'u\inv=x_1\ldots x_kw'x_k\inv \ldots x_1\inv$.
Set $w_0=w$, $w_k=w'$, and $w_i=x_i\inv w_{i-1}x_i$ for each $i\in\{1,\ldots,k-1\}$. Note that if $k\geq1$,
then $w_k=x_k\inv w_{k-1}x_k$. We claim that $w_0\ci w_1\ci\ldots\ci w_k$.
The claim is true for $k=0$ since $w\ci w'$.
Let $k\geq1$. Since
$w=x_1\ldots x_kx_k\inv \ldots x_1\inv wx_1\ldots x_kx_k\inv \ldots x_1\inv$, we have $x_1\in A(w)$.
Thus, by Lemma~\ref{lxxw}, $w\ci x_1\inv wx_1=w_1$. Thus $w_1\ci w'$, and the claim follows by induction on $k$.

The converse is true since $\ci$ is transitive.
\end{proof}

\begin{prop}
\label{psin}
Let $w=u_0e_1u_1\ldots e_mu_m\in\fr(X)$ be canonical. Then:
\begin{itemize}
  \item[\rm(1)] $w$ has finitely many conjugates of the form $x\inv wx$, where $x\in Y$;
  \item[\rm(2)] $[w]_{\ci}=\{w\}$ if and only if $A(w)=\emptyset$.
\end{itemize}
\end{prop}
\begin{proof}
Statement (1) follows from Lemma~\ref{lxxw} and the fact that $A(w)$ is finite.
If $A(w)=\emptyset$, then $[w]_{\ci}=\{w\}$ by Lemmas~\ref{lstr} and~\ref{lxxw}.
Suppose $A(w)\ne\emptyset$, and let $x\in A(w)$. Then $x\inv wx\ci w$ by Lemma~\ref{lxxw}.
Since $x\inv wx=x\inv u_0e_1u_1\ldots e_mu_mx$, we have $x\notin A(x\inv wx)$. Hence $x\inv wx\ne w$, and so
$[w]_{\ci}\ne\{w\}$.
\end{proof}

Our next objective is to prove that the conjugacy class of any $w\in\fr(X)$ is finite.

\begin{prop}
\label{pfin}
Let $e=e_1\ldots e_k\in\fr(X)$ be a canonical idempotent. Then the conjugacy class of $e$ is finite.
\end{prop}
\begin{proof}
Suppose $x\inv ex$ is a conjugate of $e$, where $x\in Y$. Then $x\in A(e)$, and so
some factor $e_i$ of $e$ must have the form $e_i=xhx\inv $, where $h$ is a canonical idempotent.
We may assume that $i=1$. Then $x\inv ex=x\inv xhx\inv e_2\ldots e_kx=hx\inv e_2\ldots e_kx$.
Since $hx\inv e_2\ldots e_kx$ is canonical with $|hx\inv e_2\ldots e_kx|=|e|$, it follows that
every element of $[e]_{\ci}$ can be expressed as a canonical idempotent that has the same letters and length
as $e$. Since there are only finitely many words that have the same letters and length, the result follows.
\end{proof}

\begin{conj}
\label{cfin}
Let $e\in\fr(X)$ be a canonical idempotent. Then the conjugacy class of $e$ has $\frac{|e|}2+1$ elements.
\end{conj}

To prove that $[w]_{\ci}$ is finite for a general $w\in\fr(X)$, we will need the following concept
and some lemmas.

\begin{defi}
\label{dtw}
Let $w\in\fr(X)$. A \emph{conjugacy tree} $T(w)$ of $w$ is defined as a rooted tree \cite[p.~188]{HaHi08}
constructed as follows:
\begin{itemize}
  \item[(a)] $w$ is the root of $T(w)$;
  \item[(b)] suppose that the vertices at level $n\geq0$ of $T(w)$ have already been constructed. We construct
the vertices at level $n+1$ as follows. For every vertex $w'$ at level $n$ and every $x\in A(w')$,
we place $x\inv w'x$ as a child of $w'$ provided $x\inv w'x$ does not already occur in $T(w)$.
\end{itemize}
For a vertex $w'$ of $T(w)$, we denote by $S_T(w')$ the rooted subtree of $T(w)$ that has $w'$ as the root
and contains all descendants of $w'$ in $T(w)$.
\end{defi}

As an example, the conjugacy tree of $w=(ab\inv ba\inv cc\inv )c\inv (ab\inv ba\inv )$ is presented in Figure~\ref{f61}.
The word $w$ is canonical with two idempotent pieces (enclosed in parentheses) and one root piece $c\inv $.
Since $A(w)=\{a,c\}$, the root $w$ has two children $w_1=a\inv wa$ and $w_2=c\inv wc$ (with edges
leading to these children labeled by $a$ and $c$). We have $A(w_1)=\{a\inv ,b\inv \}$, but $aw_1a\inv =w$, so
only $w_3=bw_1b\inv $ is placed as a child of $w_1$. We have $A(w_2)=\{c\inv \}$ and $A(w_3)=\{b\}$, but both
$cw_2c\inv =w$ and $b\inv w_3b=w_1$ already occur in $T(w)$, so the tree is completed. Since the vertices of $T(w)$
are precisely the elements of the conjugacy class of $w$ (see Lemma~\ref{ldw} below), $[w]_{\ci}$ has four elements.

\begin{figure}[ht]
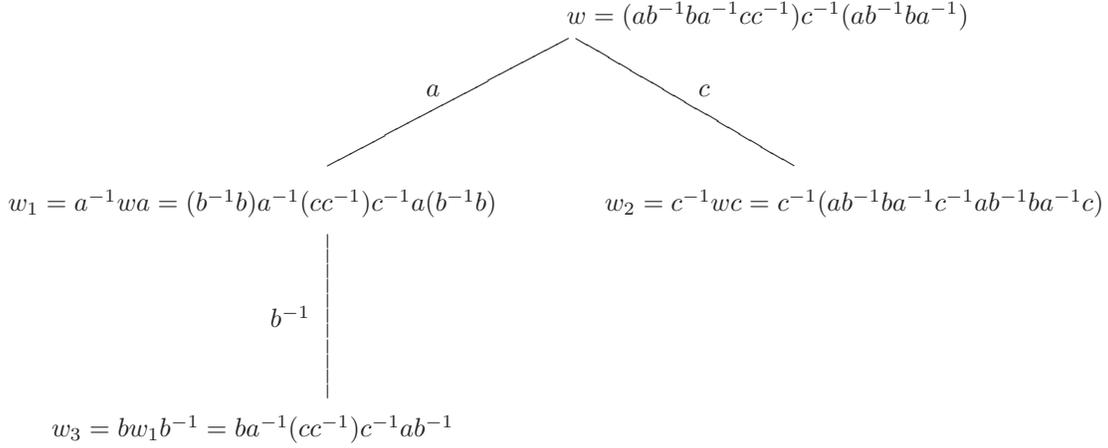

\[
\xy
(38.5,10)*{w=(ab\inv ba\inv cc\inv )c\inv (ab\inv ba\inv )}="a";
(-30,-15)*{w_1=a\inv wa=(b\inv b)a\inv (cc\inv )c\inv a(b\inv b)}="w1";
(50,-15)*{w_2=c\inv wc=c\inv (ab\inv ba\inv c\inv ab\inv ba\inv c)}="w2";
(-30,-45)*{w_3=bw_1b\inv =ba\inv (cc\inv )c\inv ab\inv  }="w1";
(-6,0)*{a}="la";
(30,0)*{c}="lc";
(-25,-30)*{b\inv }="lb";
(12,7)*{}="a1";
(13,7)*{}="a2";
(-20,-20)*{}="a5";
(-20,-10)*{}="a3";
(42,-10)*{}="a4";
(-20,-40)*{}="a6";
"a1";"a3" **\crv{} ?>* \dir{};
"a2";"a4" **\crv{} ?>* \dir{};
"a5";"a6" **\crv{} ?>* \dir{};
\endxy
\]
\caption{A conjugacy tree.}\label{f61}
\end{figure}

The following lemma follows immediately from Definition~\ref{dtw}, Lemmas~\ref{lxxw} and~\ref{lstr}, and the fact that $A(u)$ is finite
for every $u\in\fr(X)$.

\begin{lemma}
\label{ldw}
Let $w\in\fr(X)$. Then:
\begin{itemize}
  \item[\rm(1)] the vertices of $T(w)$ are precisely the elements of the conjugacy class of $w$;
  \item[\rm(2)] every vertex of $T(w)$ has finitely many children;
  \item[\rm(3)] for every vertex $w'$ of $T(w)$, all vertices of $S_T(w')$ are contained in $[w']_{\ci}$.
\end{itemize}
\end{lemma}

In view of Lemma~\ref{lmain}, we can extend the definition of the root of a canonical word
to an arbitrary $u\in Y^+$ by setting $R(u)=R(w)$, where $w$ is a canonical word that is equal to $u$ in $\fr(X)$.

\begin{lemma}
\label{lrw}
Let $w=u_0e_1u_1\ldots e_mu_m\in\fr(X)$ be canonical. Suppose $x\in A(w)$ with $x\in\{i(u_0),i(u_1)\}$
and $x\inv \in\{t(u_m),t(u_{m-1})\}$. Let $w'=x\inv wx$. Then $|R(w')|<|R(w)|$.
\end{lemma}
\begin{proof}
If $u_0=xs$ and $u_m=tx\inv $, then $w'=x\inv xse_1u_1\ldots e_mtx\inv x$, and so
$|R(w')|=|su_1\ldots u_{m-1}t|<|u_0u_1\ldots u_m|=|R(w)|$.
If $u_0=xs$ and $u_{m-1}=tx\inv $ (so $u_m$ must be $1$), then
\[
w'=x\inv xse_1u_1\ldots e_{m-1}tx\inv e_mx,
\]
and so
$|R(w')|=|su_1\ldots u_{m-2}t|<|u_0u_1\ldots u_{m-1}|=|R(w)|$. We obtain the same inequality
in the remaining two cases: $u_1=xs$, $u_m=tx\inv $ and $u_1=xs$, $u_{m-1}=tx\inv $.
\end{proof}

By a path in a rooted tree $T$, we mean a sequence $v_0,v_1,v_2,\ldots$ (finite or infinite) of vertices of $T$
such that $v_i$ is a child of $v_{i-1}$ for every $i\geq1$.

\begin{lemma}
\label{lpath}
Let $w=e_1u_1\ldots e_mu_m\in\fr(X)$ be canonical with $m\geq1$ and $xhx\inv $ a factor of~$e_1$.
Suppose $w_1=x\inv wx$ is a child of $w$ in $T(w)$. Then for every path $w_1,w_2,w_3,\ldots$ from $w_1$ in $T(w)$,
there are $x_1,x_2,x_3,\ldots$ in $Y$ such that for every $i\geq1$,
\begin{itemize}
  \item[\rm(1)] $w_i$ has canonical form $w_i=u^i_0\ldots$ with $i(u^i_0)=x_i\inv $
or $w_i=e_1^iu_1^i\ldots$ with $i(u_1^i)=x_i\inv $;
  \item[\rm(2)] if $i\geq2$, then $x_iw_ix_i\inv =w_{i-1}$ and if $w_i=e_1^iu_1^i\ldots$, then $|e_1^i|<|e_1^{i-1}|$;
  \item[\rm(3)] if $w_i=u^i_0\ldots$, then $w_i$ is a leaf in $T(w)$.
\end{itemize}
\end{lemma}
\begin{proof}
Let $w_1,w_2,w_3,\ldots$ be a path from $w_1$ in $T(w)$. We claim that (1)--(3) hold for $i=1$
with $x_1=x$. Since $xhx\inv $ is a factor of $e_1$, $e_1=xhx\inv f$, and so
\[
w_1=x\inv wx=x\inv xhx\inv fu_1\ldots=hx\inv fu_1\ldots.
\]
Suppose $h=1$. Then $w_1=u_0^1\ldots$ with $u_0^1=x\inv $ (if $f\ne1$) or $u_0^1=x\inv y\ldots$ where $y=i(u_1)$ (if $f=1$).
Moreover, $A(w_1)=\{x\inv \}$, so the only possible child of $w_1$ is $xw_1x\inv =w$. However,
since $w$ already occurs in $T(w)$, it would not have been placed as a child of $w_1$, which implies
that $w_1$ is a leaf.
Suppose $h\ne 1$. Then $w_1=e_1^1u_1^1\ldots$ with $e_1^1=h$ and $u_1^1=x\inv $ (if $f\ne1$) or
$u_1^1=x\inv y\ldots$ (if $f=1$). Since (2) holds vacuously when $i=1$, the claim has been proved.

Let $i\geq2$ and suppose there are $x_1,\ldots,x_{i-1}$ in $Y$
such that (1)--(3) are satisfied for every $j\in\{1,\ldots,i-1\}$.
If $w_{i-1}=u^{i-1}_0\ldots$, then $w_{i-1}$ is a leaf in $T(w)$ by the inductive hypothesis, so the path ends at $w_{i-1}$.

Suppose $w_{i-1}=e_1^{i-1}u_1^{i-1}\ldots$ with $x_{i-1}\inv =i(u_{i-1}^1)$. Since $w_i$ is a child of $w_{i-1}$,
there is some $x_i\in A(w_{i-1})$ such that $w_i=x_i\inv w_{i-1}x_i$. Then $x_iw_ix_i\inv =w_{i-1}$ by Lemma~\ref{lxxw}.
We claim that $x_i\ne x_{i-1}\inv $. Suppose $i=2$. Then
$x_1w_1x_1\inv =xw_1x\inv =w$, and so $x_2\ne x_1\inv $
since otherwise $w_2$ would be equal to $w$ and it would not have been placed as a child of $w_1$.
Suppose $i\geq2$. Then $x_{i-1}w_{i-1}x_{i-1}\inv =w_{i-2}$ by the inductive hypothesis, and so $x_i\ne x_{i-1}\inv $
since otherwise $w_i$ would be equal to $w_{i-2}$ and it would not have been placed as a child of $w_{i-1}$.
The claim has been proved.
Therefore, $x_i$ must be the first letter of some factor $x_ih_ix_i\inv $ of $e_1^{i-1}$, that is, $e_1^{i-1}=x_ih_ix_i\inv f_i$.
Then
\[
w_i=x_i\inv w_{i-1}x_i=x_i\inv x_ih_ix_i\inv f_iu_1^{i-1}\ldots=h_ix_i\inv f_iu_1^{i-1}\ldots.
\]
Suppose $h_i=1$. Then $w_i=u_0^i\ldots$ with $u_0^i=x_i\inv $ (if $f_i\ne1$) or $u_0^i=x_i\inv x_{i-1}\inv \ldots$ (if $f_i=1$).
Moreover, $A(w_i)=\{x_i\inv \}$, so the only possible child of $w_i$ is $x_iw_ix_i\inv =w_{i-1}$. However,
since $w_{i-1}$ already occurs in $T(w)$, it would not have been placed as a child of $w_i$, which implies
that $w_i$ is a leaf.
Suppose $h_i\ne 1$. Then $w_i=e_1^iu_1^i\ldots$ with $e_1^i=h_i$ and $u_1^i=x_i\inv $ (if $f_i\ne1$) or
$u_1^i=x_i\inv x_{i-1}\inv \ldots$ (if $f_i=1$).
Further, $|e_1^i|=|h_i|<|x_ih_ix_i\inv |\leq|e_1^{i-1}|$. We have proved that (1)--(3) hold for $i$, and the
result follows by induction.
\end{proof}

\begin{rem}
\label{rpath}
We have a dual of Lemma~\ref{lpath}.
Let $w=u_0e_1u_1\ldots e_m\in\fr(X)$ be canonical with $m\geq1$ and $xhx\inv $ a factor of $e_m$.
Suppose $w_1=x\inv wx$ is a child of $w$ in $T(w)$. Then the conclusion of Lemma~\ref{lpath} follows
with $u_0^i$, $e_1^i$, and $u_1^i$ replaced with $u_m^i$, $e_{m-1}^i$, and $u_{m-1}^i$, respectively.
\end{rem}

We can now prove the main results of this section.

\begin{theorem}
\label{tfr1}
For every $w\in\fr(X)$, the conjugacy class of $w$ is finite.
\end{theorem}
\begin{proof}
We may assume that $w=u_0e_1u_1\ldots e_mu_m$ is canonical. We proceed by induction on $|R(w)|$.
If $|R(w)|=0$, then $w$ is an idempotent, and so $[w]_{\ci}$ is finite by Proposition~\ref{pfin}.
Let $|R(w)|\geq1$ and suppose $[w']_{\ci}$ is finite for every $w'\in\fr(X)$ with $|R(w')|<|R(w)|$.
Let $w'=x\inv wx$, where $x\in A(w)$, be a child of $w$ in a conjugacy tree $T(w)$.
We want to prove that the subtree $S_T(w')$
is finite. Suppose $x\in\{i(u_0),i(u_1)\}$ and $x\inv \in\{t(u_m),t(u_{m-1}\}$. Then $|R(w')|<|R(w)|$ by Lemma~\ref{lrw}.
Thus $[w']_{\ci}$ is finite by the inductive hypothesis, and so $S_T(w')$ is also finite. Suppose $x$ is the first
letter of a factor of $e_1$ or $x\inv $ is the last letter of a factor of $e_m$. Then each path in $S_T(w')$ is finite
by Lemma~\ref{lpath} and its dual (see Remark~\ref{rpath}), and so $S_T(w')$ is finite by K\"{o}nig's Lemma \cite[Thm.~3.2]{HaHi08}
(since each level of $T(w)$ is finite by Lemma~\ref{ldw}).

Since $w$ has finitely many children, it follows that $T(w)$ is finite, and so $[w]_{\ci}$ is also finite.
\end{proof}

\begin{defi}
\label{ddec}
We say that the conjugacy problem for $\fr(X)$ is \emph{decidable}
if there is an algorithm that given any pair $(u_1,u_2)$ of words in $Y^+$, returns YES if $u_1$ and $u_2$ are conjugate in $\fr(X)$
and NO otherwise.
\end{defi}

\begin{theorem}
\label{tfr2}
The conjugacy problem in $\fr(X)$ is decidable.
\end{theorem}
\begin{proof}
It is well known that the word problem in $\fr(X)$ is decidable \cite[\S6.2]{La98}; that is,
there is an algorithm that given any pair $(u_1,u_2)$ of words in $Y^+$, returns YES if $u_1$ and $u_2$ are equal in $\fr(X)$
and NO otherwise. Call this algorithm $\mathcal{A}_1$. Poliakova and Schein \cite{PoSc05}
have described an algorithm that given any $u\in Y^+$, returns a canonical word
$\rho(u)\in Y^+$ such that $u$ and $\rho(u)$ are equal in $\fr(X)$. Call this algorithm $\mathcal{A}_2$. We will describe an algorithm $\mathcal{A}$
that solves the conjugacy problem.

Let $u_1,u_2\in Y^+$. First, $\mathcal{A}$ constructs a list of all elements of the conjugacy class $[u_1]_{\ci}$ in the following way.
\begin{enumerate}
  \item Using $\mathcal{A}_2$, our algorithm $\mathcal{A}$ calculates $w_1=\rho(u_1)$. This is the first element of the list and it is not marked.
  \item Suppose $\mathcal{A}$ has constructed a list of canonical words $w_1,\ldots,w_k$ ($k\geq1$), of which
$w_1,\ldots,w_t$ have been marked ($0\leq t\leq k$).
  \item If $t=k$, then $\mathcal{A}$ stops the calculation of the list.
  \item If $t<k$, then:
 \begin{itemize}
   \item[(a)] $\mathcal{A}$ calculates $A(w_{t+1})=\{x_1,\ldots,x_p\}$ (see Definition~\ref{daw});
   \item[(b)] using $\mathcal{A}_2$, algorithm $\mathcal{A}$ constructs the following list of canonical words:
\[
v_1=\rho(x_1\inv w_{t+1}x_1),\,\,v_2=\rho(x_2\inv w_{t+1}x_2),\ldots,v_p=\rho(x_p\inv w_{t+1}x_p);
\]
  \item[(c)] $\mathcal{A}$ applies algorithm $\mathcal{A}_1$ to check if $v_1$ is already on the list.
If not, then $\mathcal{A}$ places $v_1$ as the next element on the list. It repeats this procedure for $v_2,\ldots,v_p$,
marks $w_{t+1}$, and goes to step (2).
\end{itemize}
\end{enumerate}
The part of algorithm $\mathcal{A}$ described in (1)--(4) stops (by Theorem~\ref{tfr1}) and it
constructs a list $w_1,\ldots,w_n$ of all pairwise distinct elements
of $[u_1]_{\ci}$ (by Lemma~\ref{lstr} and Lemma~\ref{lxxw}).

Next, algorithm $\mathcal{A}$ applies algorithm $\mathcal{A}_1$ to check if $u_2$ is on the list $w_1,\ldots,w_n$. If so, it returns
YES ($u_1$ and $u_2$ are conjugate), otherwise it returns NO ($u_1$ and $u_2$ are not conjugate).
\end{proof}

\section{Conjugacy in McAllister $P$-semigroups}
\label{sec:P}
In addition to symmetric inverse semigroups and free inverse semigroups, there are other important classes
of inverse semigroups in which the conjugacy relation is worth studying.

Let $S$ be an inverse semigroup with semilattice $E$ of idempotents. We say that $S$ is \emph{$E$-unitary}
if for all $a\in S$ and $e\in E$, if $ea\in E$ then $a\in E$ \cite[\S5.9]{Ho95}. We note that
the free semigroup $\fr(X)$ is $E$-unitary.

Every $E$-unitary semigroup is isomorphic to a $P$-semigroup constructed by McAlister \cite{McA74}.
Consider a triple $(G,\mathcal{X},\mathcal{Y})$, called a \emph{McAlister triple} \cite[p.~194]{Ho95},
where $G$ is a group, $\mathcal{X}$ is a set
with a partial order relation $\leq$, and $\mathcal{Y}$ is a nonempty subset of $\mathcal{X}$ such that:
\begin{itemize}
  \item [(1)] $\mathcal{Y}$ is a lower semilattice under $\leq$, that is, if $A,B\in\mathcal{Y}$, then the greatest lower bound $A\land B$
 exists and belongs to $\mathcal{Y}$;
  \item [(2)] $\mathcal{Y}$ is an order ideal of $\mathcal{X}$, that is, if $A\in\mathcal{Y}$, $B\in\mathcal{X}$, and $B\leq A$,
then $B\in\mathcal{Y}$;
  \item [(3)] $G$ acts on $\mathcal{X}$ by automorphisms, that is, there is a mapping $(g,X)\to gX$ from
$G\times\mathcal{X}$ to $\mathcal{X}$ such that for all $g,h\in G$ and $A,B\in\mathcal{X}$,
$g(hA)=(gh)A$ and $A\leq B\quad\iff\quad gA\leq gB$;
  \item [(4)] $G\mathcal{Y}=\mathcal{X}$, and $g\mathcal{Y}\cap\mathcal{Y}\ne\emptyset$ for all $g\in G$.
\end{itemize}
Consider a set $P(G,\mathcal{X},\mathcal{Y})=\{(A,g)\in(\mathcal{Y},G):g\inv A\in\mathcal{Y}\}$
and define a multiplication on $P(G,\mathcal{X},\mathcal{Y})$ by
\begin{equation}
\label{sothe1}
(A,g)(B,h)=(A\land gB,gh).
\end{equation}
The set $P(G,\mathcal{X},\mathcal{Y})$
with multiplication \eqref{sothe1} is a semigroup, called a \emph{McAlister $P$-semigroup}. Every McAlister $P$-semigroup
is an $E$-unitary inverse semigroup, and every $E$-unitary inverse semigroup is isomorphic
to some McAlister $P$-semigroup \cite[Thm.~5.9.2]{Ho95}.

The following theorem describes $i$-conjugacy in any McAllister $P$-semigroup.

\begin{theorem}
\label{tmca}
Let $S=P(G,\mathcal{X},\mathcal{Y})$ be a McAlister $P$-semigroup.
For $(A,g),(B,h)\in S$, the following are equivalent:
\begin{itemize}
\item [\rm(a)] $(A,g)\ci(B,h)$;
\item [\rm(b)] there exists $(C,k)\in S$ such that \rm{(i)} $A=kB=C\land gC\land A$ and \rm{(ii)} $g=khk\inv$.
\end{itemize}
\end{theorem}
\begin{proof}
Suppose $(A,g)\ci(B,h)$, that is, there is $(C,k)\in S$ such that
$(C,k)\inv (A,g)(C,k)=(B,h)$ and $(C,k)(B,h)(C,k)\inv =(A,g)$. Since $(C,k)\inv =(k\inv C,k\inv )$
\cite[p.~194]{Ho95}, by straightforward calculations we obtain
\begin{align*}
(B,h)&=(k\inv C\land k\inv  A\land(k\inv g)C,k\inv gk)\,, \\
(A,g)&=(C\land kB\land(khk\inv )C,khk\inv )\,.
\end{align*}
It follows that $g=khk\inv $ (so (ii) holds), $A=C\land kB\land gC$, and $kB=C\land A\land gC$.
Thus $A\leq kB$ and $kB\leq A$, so $A=kB$. Further,
\[
A=C\land kB\land gC=C\land(C\land A\land gC)\land gC=C\land A\land gC,
\]
so (i) also holds. Conversely, suppose (i) and (ii) hold. Then
\begin{align*}
(C,k)\inv (A,g)(C,k)&=(k\inv C,k\inv )(A\land gC,gk)=(k\inv C\land k\inv (A\land gC),k\inv gk) \\
&=(k\inv (C\land A\land gC),h)=(k\inv (kB),h)=(B,h)\,.
\end{align*}
Similarly,  $(C,k)(B,h)(C,k)\inv =(A,g)$, and so $(A,g)\!\ci\!(B,h)$.
\end{proof}

\section{Factorizable inverse monoids}
\label{sec:factorizable}
We now describe $i$-conjugacy in factorizable inverse monoids, with the coset
monoid of a group as a particular example.

First, recall that for $a,b$ in an inverse semigroup $S$, the natural partial order is defined by $a\leq b$
if there exists an idempotent $e$ such that $a = eb$. Equivalently,
\[
a\leq b\iff b\inv a = a\inv a\iff a\inv b = a\inv a\iff ab\inv = aa\inv\iff ba\inv = aa\inv\,. \tag{N}
\]
A subset $A\subseteq S$ is said to be \emph{upward closed} if for all $a\in A$, $x\in S$,
$a\leq x$ implies $x\in A$.

For $a,b\in S$ with $a\ci b$, we set
\[
\mathcal{C}_{a,b} = \{ g\in S^1\mid g\inv ag= b,\ gbg\inv = a\}\,.
\]

\begin{theorem}
\label{thm:upward}
  Let $S$ be an inverse semigroup. For each $a,b\in S$ with $a\ci b$, $\mathcal{C}_{a,b}$ is upward closed
  in $S^1$.
\end{theorem}
\begin{proof}
  Let $g\in \mathcal{C}_{a,b}$ and suppose $g\leq h$ for $h\in S^1$.
  We use Proposition \ref{prp:sapir} and (N) to obtain: $h\inv ah = h\inv\cdot gbg\inv\cdot h =
  g\inv gbg\inv g = b$ and $hbh\inv = h\cdot g\inv ag\cdot h\inv = gg\inv a gg\inv = a$.
  Thus $h\in \mathcal{C}_{a,b}$.
\end{proof}

An inverse monoid $S$ is \emph{factorizable} if $S = E(S)\,U(S)$.
In other words, each element $a\in S$ can be written in the form $a = eg$ for some idempotent $e\in E(S)$
and some unit $g\in U(S)$.

For example, let $G$ be a group and let $\mathcal{CM}(G) = \{ Ha\mid H\leq G,\ a\in G\}$ be the
\emph{coset monoid} of $G$, where the multiplication on right cosets is defined by
$Ha\ast Kb = (H\lor aKa\inv)ab$, where $H\lor aKa\inv$ is the smallest subgroup of $G$
that contains the subgroups $H$ and $aKa\inv$.
This is a factorizable inverse monoid \cite{Schein66}, and every inverse semigroup
embeds in the coset monoid of some group \cite{McAlister80}. In this case,
$E(S)$ is the set of all subgroups of $G$ and $U(S)$ is the set of all singletons from $G$,
that is, cosets of the trivial subgroup \cite{East06}.

\begin{corollary}
\label{cor:factorizable}
  Let $S$ be a factorizable inverse monoid. Then for all $a,b\in S$, $a\ci b$ if and only if
  $a\cu b$.
\end{corollary}
\begin{proof}
  We already noted in \eqref{eq:unit_i} that $\cu\ \subseteq\ \ci$ in any inverse monoid.
  For the other inclusion, suppose $a\ci b$ and let $h\in \mathcal{C}_{a,b}$ be given.
  Then there exist $e\in E(S)$, $g\in U(S)$ such that $h = eg$.
  But then $h\leq g$. By Theorem \ref{thm:upward}, $g\in \mathcal{C}_{a,b}$.
\end{proof}

\begin{rem}
  If $X$ is a finite set, then the symmetric inverse monoid $\mi(X)$ is factorizable, so
  Corollary \ref{cor:factorizable} gives another proof of the equivalence of
  parts (i) and (iii) of Proposition \ref{pper}.
\end{rem}

In particular, we have the following.

\begin{corollary}
\label{cor:coset}
  Let $G$ be a group and let $\mathcal{CM}(G)$ be the coset monoid of $G$. For $Ha,Kb\in \mathcal{CM}(G)$,
  if $Ha\ci Kb$ if and only if there exists $g\in G$ such that $g\inv Hag = Kb$ and $gKbg\inv = Ha$.
\end{corollary}

\section{Green's relations and Clifford semigroups}
\label{sec:clifford}
Let $S$ be a semigroup.
For $a,b\in S$, we say that $a\greenL b$ if $S^1a=S^1b$, $a\greenR b$ if $aS^1=bS^1$,
and $a\greenJ b$ if $S^1aS^1=S^1bS^1$. We set $\greenH = \greenL\cap \greenR$. We also
define $\greenD$ to be the join of $\greenL$ and $\greenR$, that is, the smallest equivalence relation
on $S$ containing both $\greenL$ and $\greenR$; it turns out that
$\greenD=\greenL\circ\greenR=\greenR\circ\greenL$ \cite[p.~46]{Ho95}. These equivalences,
called \emph{Green's relations}, play an important role in semigroup theory \cite[{\S}2.1]{Ho95}.

The conjugacy $\ci$ is always included in $\greenD$ (and so in $\greenJ$).

\begin{prop}
\label{pcid}
Let $S$ be an inverse semigroup. Then $\ci\,\,\subseteq\gd$.
\end{prop}
\begin{proof}
Let $a,b\in S$ with $a\ci b$. By Proposition~\ref{prp:sapir}, there exists $g\in S^1$ such that
$g\inv\cdot ag = b$, $ag = gb$, and $a\cdot gg\inv = a$.
Since $ag\cdot g\inv = a$, we have $a\greenR ag$. Since $g\inv\cdot ag = b$ and $gb = ag$, we have
$ag\greenL b$. Hence $a\,(\greenR\circ\greenL)\,b$, that is, $a\greenD b$.
\end{proof}

For each element $a$ in an inverse semigroup $S$, the unique
idempotent in the $\greenL$-class of $a$ is $a\inv a$, and the unique idempotent in
the $\greenR$-class of $a$ is $aa\inv$. These idempotents are conjugate.

\begin{lemma}
\label{lem:xx'x'x}
  Let $S$ be an inverse semigroup. For all $x\in S$, $xx\inv \ci x\inv x$.
\end{lemma}
\begin{proof}
  This is immediate from $x\inv\cdot xx\inv\cdot x = x\inv x$ and
  $x\cdot x\inv x\cdot x\inv = xx\inv$.
\end{proof}

In addition, $i$-conjugacy of elements implies $i$-conjugacy of their corresponding
$\greenL$-related and $\greenR$-related idempotents.

\begin{lemma}\label{lem:same_g}
  Let $S$ be an inverse semigroup and let $a,b\in S$ satisfy $a\ci b$. Then
  $a\inv a\ci b\inv b$ and $aa\inv \ci bb\inv$. More precisely, if $g\in \mathcal{C}_{a,b}$,
  then $g\in \mathcal{C}_{a\inv a,b\inv b}$ and $g\in \mathcal{C}_{aa\inv,bb\inv}$.
\end{lemma}
\begin{proof}
 Let $g\in \mathcal{C}_{a,b}$. Then
 \begin{align*}
   b\inv b &= (g\inv ag)\inv g\inv ag = g\inv a\inv \underbrace{gg\inv a}g = g\inv a\inv ag\qquad \text{and} \\
   b b\inv &= g\inv ag (g\inv ag)\inv = g\inv \underbrace{agg\inv} a\inv g = g\inv aa\inv g\,,
 \end{align*}
 using Proposition \ref{prp:sapir} in both calculations. The equalities
 $g b\inv bg\inv = a\inv a$ and $g bb\inv g\inv = aa\inv$ follow similarly.
  \end{proof}

Let $S$ be a semigroup. For $a\in S$, denote by $H_a$ the $\greenH$-class containing $a$. Any $\greenH$-class of $S$
containing an idempotent is a maximal subgroup of $S$.
An element $a\in S$ is \emph{completely regular} (or a \emph{group element}) if
its $\greenH$-class $H_a$ is a group. If $S$ is an inverse semigroup, the unique inverse $a\inv$ of
a completely regular element $a$ is also the inverse of $a$ in $H_a$, and so in particular,
$aa\inv = a\inv a$. Conversely, if $aa\inv = a\inv a$, then $a\greenH aa\inv$, so $H_a$ is a group.

\begin{lemma}\label{lem:CR}
 Let $S$ be an inverse semigroup and let $a,b\in S$ satisfy $a\ci b$. The following are
 equivalent:
 \begin{itemize}
   \item[\rm{(a)}] $a$ is completely regular;
   \item[\rm{(b)}] $b$ is completely regular;
   \item[\rm{(c)}] there exists $h\in \mathcal{C}_{a,b}$ such that $a\greenR h\greenL b$.
 \end{itemize}
\end{lemma}
\begin{proof}
(a)$\iff$(b) follows from Lemma \ref{lem:same_g}.

Assume (a), (b), and fix $g\in \mathcal{C}_{a,b}$. Set $h = gb = ag$. Then
\begin{equation}\label{eq:CR1}
hh\inv = \underbrace{agg\inv} a\inv = aa\inv\quad\text{and}\quad
h\inv h = b\inv \underbrace{g\inv g b} = b\inv b,
\end{equation}
using Proposition \ref{prp:sapir} in both calculations.
Thus $h\cdot h\inv a=aa\inv a=a$ and $bh\inv\cdot h=bb\inv b=b$, so
$a\greenR h\greenL b$. Next, $h\inv ah = g\inv a\inv aag = g\inv aa\inv ag = g\inv ag = b$, using the complete regularity of $a$ in the second equality. Similarly, $hbh\inv = gbbb\inv g\inv = gbb\inv bg\inv = gbg\inv = a$, using the complete regularity of $b$. Thus $h\in \mathcal{C}_{a,b}$. We have proven (a),(b)$\implies$(c).

Now assume (c). We have $h\inv ah=b$ and $hbh\inv = a$. Moreover since each $\gr$-class and each $\gl$-class
in an inverse semigroup contains exactly one idempotent \cite[Thm.~5.1.1]{Ho95}, we also have
$hh\inv = aa\inv$ and $h\inv h = b\inv b$. We thus compute
\[
b\inv b = h\inv h = h\inv \underbrace{hh\inv} h = h\inv \underbrace{a}a\inv  h =
\underbrace{h\inv h b} \underbrace{h\inv a\inv h} = b (h\inv a h)\inv = bb\inv\,,
\]
using Proposition \ref{prp:sapir} in the fifth equality. Therefore $b$ is completely regular, that is, (b) holds.
\end{proof}

\begin{prop}\label{prp:group_conj}
  Let $S$ be an inverse semigroup and let $a,b\in S$ satisfy $a\ci b$.
 If $a,b$ lie in the same group $\greenH$-class $H$, then
  $a$ and $b$ are group conjugate in $H$.
\end{prop}
\begin{proof}
  Since $H$ is a group, both $a$ and $b$ are completely regular, and so by
  Lemma \ref{lem:CR}, there exists $h\in \mathcal{C}_{a,b}$ such that $a\greenR h\greenL b$.
  Since $a\greenH b$, we have $h\greenH a$, that is, $h\in H$.
\end{proof}

If every element of a semigroup $S$ is completely regular, we say that $S$ is a
\emph{completely regular semigroup}. A semigroup that is both inverse and completely regular
is called a \emph{Clifford semigroup}. One can characterize Clifford semigroups
in several ways, some of which will be useful in what follows.

\begin{prop}{\rm\hskip -1mm \cite[Thm.~4.2.1]{Ho95}}
\label{prp:clifford}
  Let $S$ be an inverse semigroup. The following are equivalent:
  \begin{itemize}
    \item[\rm{(a)}] $S$ is a Clifford semigroup;
    \item[\rm{(b)}] for all $a\in S$, $aa\inv = a\inv a$;
    \item[\rm{(c)}] for all $a\in S$, $e\in E(S)$, $ea=ae$;
    \item[\rm{(d)}] $\greenL = \greenR = \greenH$.
  \end{itemize}
\end{prop}

\begin{theorem}
\label{thm:Clifford_H}
Let $S$ be a Clifford semigroup. Then for all $a,b\in S$, $a\ci b$ if and only if $a$ and $b$ belong
to the same $\greenH$-class $H$ and they are group conjugate in $H$.
\end{theorem}
\begin{proof}
  The ``if'' direction is clear. For the converse, if $a\ci b$, then
  by Lemma \ref{lem:CR} and Proposition \ref{prp:clifford}(d), there exists $h\in \mathcal{C}_{a,b}$ such that $a\greenH h\greenH b$. The rest follows from Proposition \ref{prp:group_conj}.
\end{proof}

Using $i$-conjugacy, we can give new characterizations of Clifford semigroups in the class of
inverse semigroups.

\begin{theorem}
\label{thm:Clifford_char}
Let $S$ be an inverse semigroup. The following are equivalent:
\begin{itemize}
\item[\rm(a)] $S$ is a Clifford semigroup;
\item[\rm(b)] $\ci\,\,\subseteq\greenH$;
\item[\rm(c)] $\ci\,\,\subseteq\greenR$;
\item[\rm(d)] $\ci\,\,\subseteq\greenL$;
\item[\rm(e)] no two distinct idempotents in $S$ are conjugate.
\end{itemize}
\end{theorem}
\begin{proof}
We have (a)$\implies$(b) by Theorem~\ref{thm:Clifford_H}. The implications (b)$\implies$(c)
and (b)$\implies$(d) follow from $\greenH\subseteq\greenR$ and $\greenH\subseteq\greenL$.
We have (c)$\implies$(e) and (d)$\implies$(e) by the fact that every $\greenR$-class
and every $\greenL$-class of an inverse semigroup contains exactly one idempotent.
Finally, suppose (e) holds. For $a\in S$, $aa\inv$ and $a\inv a$ are idempotents and we have
$aa\inv\ci a\inv a$ by Lemma \ref{lem:xx'x'x}. Thus $aa\inv = a\inv a$. Then (a) follows
from Proposition \ref{prp:clifford}. This completes the proof.
\end{proof}

Recall that a group $G$ is abelian if and only if the conjugacy $\cgp$ is the identity relation.
This generalizes to inverse semigroups.

\begin{theorem}\label{thm:commutative}
  Let $S$ be an inverse semigroup. Then $S$ is commutative if and only if $\ci$ is the identity relation.
\end{theorem}
\begin{proof}
  Every commutative inverse semigroup is Clifford. On the other hand, if $\ci$ is the identity relation,
  then $S$ is Clifford by Theorem \ref{thm:Clifford_char}. Thus we may assume from the outset that $S$ is a Clifford semigroup. The desired result then follows from the following chain of
  equivalences: $S$ is commutative if and only if each $\greenH$-class is an abelian
  group if and only if group conjugacy within each $\greenH$-class is
  the identity relation if and only if $\ci$ is the identity relation (by Theorem \ref{thm:Clifford_H}).
\end{proof}

Looking at conditions (b), (c) and (d) of Theorem \ref{thm:Clifford_char}, it is natural to ask what can be said
if the opposite inclusions hold. We conclude this section with two results that answer this question.

\begin{theorem}\label{thm:semilattice}
  Let $S$ be an inverse semigroup. The following are equivalent:
  \begin{itemize}
    \item[\rm(a)] $S$ is a semilattice;
    \item[\rm(b)] $\greenL\subseteq\ \ci$;
    \item[\rm(c)] $\greenR\subseteq\ \ci$.
  \end{itemize}
\end{theorem}
\begin{proof}
  In a semilattice, $\greenL$ and $\greenR$ are trivial, so (a)$\implies$(b) and (a)$\implies$(c) follow.
  Assume (c). For each $a\in S$, $aa\inv \greenR a$, and so $g\inv aa\inv g = a$ for some $g\in S^1$. But
  every conjugate of an idempotent is an idempotent, so each $a\in S$ is idempotent. Thus (a) holds.
  The proof of (b)$\implies$(a) is similar.
\end{proof}

A semigroup is said to be $\greenH$-\emph{trivial} if $\greenH$ is the identity relation.

\begin{theorem}\label{thm:H-trivial}
  Let $S$ be an inverse semigroup. Then $\greenH\subseteq\ \ci$ if and only if $S$ is $\greenH$-trivial.
\end{theorem}
\begin{proof}
The ``if'' direction is obvious, so assume $\greenH\subseteq\ \ci$. If $H$ is a group $\greenH$-class, then by Proposition \ref{prp:group_conj}, all elements of $H$ are group conjugate, hence $H$ is a trivial subgroup. Now suppose $a\greenH b$.
We compute
\[
(ba\inv)\inv ba\inv = ab\inv ba\inv = aa\inv aa\inv = aa\inv = bb\inv = bb\inv bb\inv = ba\inv ab\inv = ba\inv (ba\inv)\inv\,.
\]
Thus $ba\inv$ is completely regular, that is, it is in some group $\greenH$-class $H$.
By the above computation, $aa\inv$ is the identity in $H$.
Since $H$ is trivial, $ba\inv = aa\inv$. By (N), we have $a\leq b$. Repeating the argument with the roles of $a$ and $b$ reversed,
we also obtain $b\leq a$.
Thus $a=b$. Therefore $\greenH$ is the identity relation as claimed.
\end{proof}

\section{The Bicyclic Monoid and Stable Inverse Semigroups}
\label{sec:stable}

The \emph{bicyclic monoid} $\Bi$, which is an inverse semigroup, is usually defined in terms of a monoid presentation
$\langle x,y\mid xy = 1\rangle$ \cite[p.~32]{Ho95} \cite[Sect.~3.4]{La98}. It has a more convenient isomorphic realization
as the set $\Bi$ of ordered pairs of nonnegative integers with the following multiplication:
\[
(a, b) (c, d) = (a - b + \max(b, c), d - c + \max(b, c))\,.
\]
For any $(a,b)\in \Bi$, $(a,b)\inv = (b,a)$ and $(a,b)$ is an idempotent if and only if $a = b$. The smallest group congruence $\relsigma$ in $\Bi$ is characterized as follows \cite[p. 101]{La98}:
\[
(a,b)\relsigma (c,d) \iff a-b=c-d.
\]

\begin{theorem}\label{thm:bicyclic}
In $\Bi$, $\ci\ =\ \relsigma$.
\end{theorem}
\begin{proof}
Suppose $(a,b)\ci (c,d)$. Then for some $(e,f)\in \Bi$, we have $(a,b)=(f,e)(c,d)(e,f)$. Expanding this,
we get
\begin{align*}
  a &= f - e - d + c + m \\
  b &= f - e + m
\end{align*}
where $m = \max(d-c+\max(e,c),e)$. Thus $a - b = c - d$, that is, $(a,b)\relsigma (c,d)$.

Conversely, suppose that $(a,b)\relsigma (c,d)$, so $a-b=c-d$. We claim that for
$x=\min(c,d)$ and $y=\min(a,b)$, we have
\begin{equation}\label{Eq:sigma_conj}
  (a,b) = (y,x)(c,d)(x,y)\quad\text{and}\quad (c,d) = (x,y)(a,b)(y,x)\,.
\end{equation}
To prove this, we compute
\begin{align*}
  (y,x)(c,d)(x,y) &= (y,x)(c-d+\underbrace{\max(d,x)}_d,y-x+\underbrace{\max(d,x)}_d) \\
  &= (y,x)(c,y-x+d) \\
  &= (y-x+\underbrace{\max(x,c)}_c,y-x+d-c+\underbrace{\max(x,c)}_c) \\
  &=(y-x+c,y-x+d)\,,
\end{align*}
and similarly,
\begin{align*}
(x,y)(a,b)(y,x)&= (x,y)(a-b + \underbrace{\max(b,y)}_b, x-y+\underbrace{\max(b,y)}_b) \\
    &= (x,y)(a,x-y+b) \\
    &= (x-y+\underbrace{\max(y,a)}_a, x-y_b-a+\underbrace{\max(y,a)}_a) \\
    &= (x-y+a, x-y+b)\,.
\end{align*}
Comparing the results of these calculations, and noting that $c-a=d-b$, we see that we will have established \eqref{Eq:sigma_conj} once we have proven $x - y = c - a = d - b$.

Observe that $a-b=c-d$ implies that $a\leq b \iff c\leq d$, and also $a\geq b\iff c\geq d$. Thus
$x = c\iff y = a$ and $x = d\iff y = b$. In the former case, $x-y = c-a$ and in the latter case, $x-y = d-b$.
In both cases, we have $c-a=d-b$ because $a-b=c-d$. This completes the proof.
\end{proof}

A semigroup $S$ is \emph{left stable} if, for all $a,b\in S$, $S^1 a \subseteq S^1 ab$ implies $S^1 a = S^1 ab$,
that is, $a\greenL ab$. This can be equivalently formulated as $a\in S^1 ab$ implies $ab\in S^1 a$ for all $a,b\in S$.
\emph{Right stability} is defined dually, and a semigroup is said to be \emph{stable} if it is both left and right stable \cite[p.~31]{CP}.
Every periodic semigroup, and in particular every finite semigroup, is stable. In inverse semigroups, left
and right stability are equivalent. We also have a useful characterization, which in fact holds more generally for
regular semigroups.

\begin{prop}[\cite{RS}, Ex. A.2.2(8), p. 595]
\label{prp:avoids}
  Let $S$ be an inverse semigroup. Then $S$ is stable if and only if $S$ does not contain
  an isomorphic copy of the bicyclic monoid as a subsemigroup.
\end{prop}

Here we give a new characterization of stability in terms of $i$-conjugacy and the natural partial order.

\begin{theorem}\label{thm:stable}
  An inverse semigroup $S$ is stable if and only if $\ci\cap\leq$ is the identity relation on $S$.
\end{theorem}
\begin{proof}
  Assume $S$ is stable. Suppose $a\ci b$ and $a\leq b$. Then $g\inv ag = b$ and $gbg\inv = a$ for some $g\in S^1$.
  First, we compute
  \begin{equation}\label{eq:stable1}
    a = a\underbrace{a\inv a} = \underbrace{aa\inv} b = ab\inv b = ab\inv g\inv gb = a(\underbrace{gb})\inv \underbrace{gb}
    = a(ag)\inv ag\,,
  \end{equation}
  where the second and third equalities follow from (N) and the fourth and sixth equalities follow from Proposition \ref{prp:sapir}. Thus $a\in S^1 ag$. Since $S$ is stable, $ag\in S^1 a$, that is, $ag = ca$ for some $c\in S^1$.
  Now
  \[
  ag = ca = \underbrace{ca}a\inv a = \underbrace{ag}a\inv a = g\underbrace{ba\inv} a = gaa\inv a = ga\,,
  \]
  where the fourth equality follows from Proposition \ref{prp:sapir} and the fifth equality follows from (N).
  Using this in \eqref{eq:stable1}, we have
  \begin{align*}
    a &= a(ga)\inv ga = \underbrace{aa\inv} g\inv ga = b\underbrace{a\inv g\inv} ga = b(\underbrace{ga})\inv \underbrace{ga}\\
     &= \underbrace{b}(ag)\inv ag = g\inv \underbrace{ag(ag)\inv ag} = g\inv ag = b\,,
  \end{align*}
  where the third equality follows from (N).

  Conversely, suppose that $S$ is not stable, so by Proposition \ref{prp:avoids}, $S$ contains a copy of the bicyclic monoid $\Bi$.
  For nonnegative integers $m$, $n$ with $m < n$, we have $(n,n) < (m,m)$ in $\Bi$. By Theorem \ref{thm:bicyclic},
  $(m,m)\ci (n,n)$. Thus $\ci\cap\leq$ strictly contains the identity relation.
\end{proof}

\begin{corollary}\label{cor:finite}
  Let $S$ be a finite inverse semigroup. Then $\ci\cap\leq$ is the identity relation.
\end{corollary}

\section{Problems}
\label{sec:problems}

Almost factorizable inverse semigroups naturally generalize factorizable inverse monoids in the sense that
an inverse monoid is almost factorizable if and only if it is factorizable \cite{La98}.

\begin{prob}
  Does $i$-conjugacy in almost factorizable inverse semigroups have a reasonable characterization
  suitably generalizing {\rm Corollary \ref{cor:factorizable}?}
\end{prob}

The basic definition \eqref{eq:inv_conj} of conjugacy in an inverse semigroup can,
in principle, be extended to any class of semigroups in which there is some natural notion
of unary (weak) inverse map. For example, let $(S,\cdot,{}')$ be a unary $E$-inversive
semigroup, that is, $(S,\cdot)$ is a semigroup and the identity $x'xx'=x'$ holds.
Unary $E$-inversive semigroups include unary regular semigroups (in which the identity
$xx'x=x$ also holds) and epigroups (in which $x\mapsto x'$ is the unique pseudoinverse
\cite{Shevrin}). Define a notion of conjugacy in such unary semigroups by
$a\!\sim\! b$ if $g'ag=b$ and $gbg'=a$ for some $g\in S^1$. In general, these relations
will not be transitive (except, for instance, when the identity $(xy)'=y'x'$ holds), so
it is necessary to consider the transitive closure $\sim^*$.

\begin{prob}
  Study this notion of conjugacy in various interesting subclasses of unary $E$-inversive semigroups.
\end{prob}

Somewhat more promising is to consider the alternative formulations of $i$-conjugacy given
in Proposition~\ref{prp:sapir}. For instance, part (c) of the proposition depends only on the
idempotents $gg\inv$ and $g\inv g$. This immediately suggests a generalization to
\emph{restriction semigroups} and their various specializations, such as \emph{ample} semigroups
(see \cite{Gould} and the references therein). An algebra $(S,\cdot,{}^+,{}^*)$ is a
restriction semigroup if $(S,\cdot)$ is a semigroup; $S\to S;x\mapsto x^+$ is a unary
operation satisfying $x^+ x = x$, $x^+ y^+ = y^+ x^+$, $(x^+ y)^+ = x^+ y^+$,
$(xy)^+x = xy^+$; $S\to S;x\mapsto x^*$ is a unary operation satisfying dual identities;
and $(x^+)^* = x^+$, $(x^*)^+ = x^*$. Here $x^+$ and $x^*$ turn out to be idempotents.
Any inverse semigroup is a restriction semigroup with $x^+ = xx\inv$ and $x^* = x\inv x$.

For $a,b$ in a restriction semigroup $S$, define
\[
a \sim_r b\quad\iff\quad \exists_{g\in S^1}\ (\,ag=gb,\ ag^+ = a,\ g^*b = b\,)\,.
\]

\begin{prob}
  Study $\sim_r$ in restriction and ample semigroups. What is the relationship between
  $\sim_r$ and $\ncon${\rm?}
\end{prob}



\begin{thebibliography}{99}

\bibitem{ArKiKoMa17}
J. Ara\'{u}jo, M. Kinyon, J. Konieczny, and A. Malheiro,
Four notions of conjugacy for abstract semigroups,
\textit{Proc. Roy. Soc. Edinburgh Sect. A},
\textbf{147} (2017) , 1169--1214.

\bibitem{AKM}
J. Ara\'ujo, J. Konieczny, and A. Malheiro,
Conjugation in semigroups,
\textit{J. Algebra} \textbf{403} (2014), 93--134.

\bibitem{CP}
A.H. Clifford and G.B. Preston,
\textit{The Algebraic Theory of Semigroups}. Vol. II.
Mathematical Surveys \textbf{7}, American Mathematical Society,
Providence, R.I. 1967.

\bibitem{DiHa03}
M. Dieng, T. Halverson, and V. Poladian,
Character formulas for $q$-rook monoid algebras,
\textit{J. Algebraic Combin.} \textbf{17} (2003), 99--123.

\bibitem{DuFo04}
D.S. Dummit and R.M. Foote,
\textit{Abstract Algebra},
Third Edition, John Wiley \& Sons, 2004.

\bibitem{East06}
J. East,
Factorizable inverse monoids of cosets of subgroups of a group,
\textit{Comm. Algebra} \textbf{34} (2006), 2659--2665.

\bibitem{Gould}
V. Gould,
Notes on restriction semigroups and related structures, formerly
(weakly) left $E$-ample semigroups,
preprint, \url{http://www-users.york.ac.uk/~varg1/restriction.pdf}.

\bibitem{HaHi08}
J.M. Harris, J.L. Hirst, and M.J. Mossinghoff,
\textit{Combinatorics and Graph Theory},
Second Edition, Springer, New York, 2008.

\bibitem{Hi06}
P.M. Higgins,
The semigroup of conjugates of a word,
\textit{Int.\ J. Algebra Comput.} \textbf{16} (2006), 1015--1029.

\bibitem{Ho95}
J.M. Howie,
\textit{Fundamentals of Semigroup Theory},
Oxford University Press, New York, 1995.

\bibitem{HrJe99}
K. Hrbacek and T. Jech,
\textit{Introduction to Set Theory\/},
Third Edition, Taylor \& Francis, New York, 1999.

\bibitem{Ko13}
J. Konieczny,
Centralizers in the infinite symmetric inverse semigroup,
\textit{Bull. Austral. Math. Soc.} \textbf{87} (2013), 462--479.

\bibitem{Ko18}
J. Konieczny,
A new definition of conjugacy for semigroups,
\textit{J. Algebra and Appl.} {\bf 17}, 1850032 (2018) [20 pages].

\bibitem{KuMa07}
G. Kudryavtseva and V. Mazorchuk,
On conjugation in some transformation and Brauer-type semigroups,
\textit{Publ.\ Math.\ Debrecen} \textbf{70} (2007), 19--43.

\bibitem{KuMa09}
G. Kudryavtseva and V. Mazorchuk,
On three approaches to conjugacy in semigroups,
\textit{Semigroup Forum} \textbf{78} (2009), 14--20.

\bibitem{La79} G. Lallement,
\textit{Semigroups and Combinatorial Applications\/},
John Wiley \& Sons, New York, 1979.

\bibitem{La98}
M.V. Lawson,
\textit{Inverse semigroups. The theory of partial symmetries},
World Scientific Publishing, River Edge, NJ, 1998.

\bibitem{Le91}
I. Levi,
Normal semigroups of one-to-one transformations,
\textit{Proc. Edinburgh Math. Soc.} \textbf{34} (1991), 65--76.

\bibitem{Li53}
A.E. Liber,
On symmetric generalized groups, (Russian),
\textit{Mat. Sbornik N.S.} \textbf{33(75)} (1953), 531--544.

\bibitem{Li96}
S. Lipscomb,
\textit{Symmetric Inverse Semigroups},
Mathematical Surveys and Monographs \textbf{46},
American Mathematical Society, Providence, RI, 1996.

\bibitem{LySc77}
R.C. Lyndon and P.E. Schupp,
\textit{Combinatorial Group Theory},
Springer-Verlag, New York, 1977.

\bibitem{McA74}
D.B. McAlister,
Groups, semilattices and inverse semigroups, I, II,
\textit{Trans.\ Amer.\ Math.\ Soc.} \textbf{192} (1974), 227--244; ibid.\ \textbf{196} (1974), 351--370.

\bibitem{McAlister80}
D.B. McAlister,
Embedding inverse semigroups in coset semigroups,
\textit{Semigroup Forum} \textbf{20} (1980), 255--267.

\bibitem{Ot84}
F. Otto,
Conjugacy in monoids with a special Church-Rosser presentation is decidable,
\textit{Semigroup Forum} \textbf{29} (1984), 223--240.

\bibitem{Pe84}
M. Petrich,
\textit{Inverse Semigroups},
John Wiley \& Sons, New York, 1984.

\bibitem{PoSc05}
O. Poliakova and B.M. Schein,
A new construction for free inverse semigroups,
\textit{J. Algebra} \textbf{288} (2005), 20--58.

\bibitem{RS}
J. Rhodes and B. Steinberg,
\textit{The q-theory of finite semigroups},
Springer Monographs in Mathematics. Springer, New York, 2009.

\bibitem{Sapir}
M. Sapir,
\url{http://mathoverflow.net/questions/52107/the-concept-conjugate-class-in-monoids}.

\bibitem{Sc64}
W.R. Scott,
\textit{Group Theory},
Prentice-Hall, Englewood Cliffs, New Jersey, 1964.

\bibitem{Schein66}
B.M. Schein,
Semigroups of strong subsets,
(Russian) \textit{Vol\v{z}. Mat. Sb.} vyp. \textbf{4} (1966), 180--186.

\bibitem{Shevrin}
L. N. Shevrin,
Epigroups, in
V. B. Kudryavtsev and I. G. Rosenberg (eds.),
\emph{Structural Theory of Automata, Semigroups, and Universal Algebra},
331--380,
NATO Sci. Ser. II Math. Phys. Chem., 207, Springer, Dordrecht, 2005.

\bibitem{yam}
A. Yamamura,
Locally full HNN extensions of inverse semigroups,
\textit{J. Australian Math. Soc.} \textbf{70} (2001), 235--272.

\bibitem{Zh91}
L. Zhang,
Conjugacy in special monoids,
\textit{J. Algebra} \textbf{143} (1991), 487--497.

\bibitem{Zh92}
L. Zhang,
On the conjugacy problem for one-relator monoids with elements of finite order,
\textit{Int. J. Algebra Comput.} \textbf{2} (1992), 209--220.

\end{thebibliography}
\end{document}